\documentclass[preprint,12pt]{elsarticle}

\usepackage{amsmath,amssymb,amsfonts,amsthm,mathrsfs}
\usepackage{mathtools}

\usepackage{longtable}
\usepackage{booktabs}

\usepackage{anyfontsize}
\usepackage{graphicx}
\usepackage{subcaption}
\usepackage{float}
\usepackage{afterpage}

\usepackage{seqsplit}
\usepackage{xcolor}
\usepackage{enumitem}
\usepackage{geometry}
\usepackage[hidelinks]{hyperref}
\usepackage{bookmark}


\usepackage[capitalize]{cleveref}

\newtheorem{theorem}{Theorem}[section]
\newtheorem{lemma}[theorem]{Lemma}
\newtheorem{corollary}[theorem]{Corollary}
\newtheorem{definition}[theorem]{Definition}

\theoremstyle{definition}

\newtheorem{remark}{Remark}[section]

\theoremstyle{plain}
\newtheorem{assumption}[theorem]{Assumption}

\journal{Journal of Differential Equations}

\begin{document}

\begin{frontmatter}

\title{Persistence and long-time breakdown of most probable paths under
time-dependent fractional noise with applications to KAM tori}

\author[inst1]{Yanbin Zhu}
\ead{zhuyb26@mails.jlu.edu.cn}

\author[inst1]{Xiaomeng Jiang}
\ead{jxmlucy@hotmail.com}

\author[inst1]{Yong Li\corref{cor1}}
\ead{liyong@jlu.edu.cn}

\cortext[cor1]{Corresponding author.}

\affiliation[inst1]{
    organization={College of Mathematics, Jilin University},
    city={Changchun},
    postcode={130012},
    country={China}
}

\begin{abstract}
We investigate the persistence of most probable paths through the
Onsager--Machlup functional for multidimensional stochastic differential
equations driven by fractional Brownian motion with time-dependent diffusion
coefficients and Hurst parameter $H\in(1/4,1)$. Under suitable structural and
variational conditions, deterministic trajectories remain most probable paths
for sufficiently small noise in both the fixed-endpoint transition problem and
the free-endpoint evolution problem, whereas sufficiently large noise destroys
their local minimality. More generally, when exact persistence does not hold,
global most probable paths converge to the corresponding trajectories of the
noise-free system in both the uniform and H\"older topologies at the rate
$O(\epsilon)$. We further analyze the second variation along periodic
deterministic trajectories over time intervals of length $NT$. For $H>1/2$,
positive definiteness, and hence local minimality, is lost on sufficiently long
intervals. For $H\in(1/4,1/2]$, long-time positive definiteness holds for the
fixed-endpoint problem, but this conclusion does not directly extend to the
free-endpoint setting. We also establish the persistence of KAM tori in nearly
integrable Hamiltonian systems in the sense of most probable evolution paths.
Finally, a two-dimensional numerical example illustrates the persistence of
deterministic trajectories under small noise and their pronounced deviation
under large noise.
\end{abstract}

\end{frontmatter}

\allowdisplaybreaks

\section{Introduction}
Compared with ordinary differential equations, stochastic differential
equations incorporate random fluctuations into the evolution of dynamical
systems and may therefore generate a wide range of possible trajectories.
This naturally raises a fundamental question: among all admissible
trajectories, which paths are the most likely to occur? The
Onsager--Machlup functional provides a variational framework for addressing
this question by quantifying the relative concentration of the stochastic
system around prescribed reference paths. The main objective of this paper is
to determine when a deterministic trajectory of the underlying ordinary
differential equation retains its role as a most probable path after
fractional stochastic perturbations are introduced.

More precisely, we consider the following $n$-dimensional stochastic
differential equation driven by time-dependent fractional noise:
\begin{equation}
\label{fir}
X_t
=
x_0
+
\int_0^t b_s(X_s)\,ds
+
\int_0^t \sigma_s\,dB_s^H,
\end{equation}
where
\begin{equation*}
b:\mathbb R^+\times\mathbb R^n\longrightarrow\mathbb R^n
\end{equation*}
is the drift coefficient and
\begin{equation*}
\sigma_s
=
\operatorname{diag}
\bigl(
\sigma_s^1,\ldots,\sigma_s^n
\bigr)
\end{equation*}
is a time-dependent diffusion matrix. We assume that each
$\sigma^i:\mathbb R^+\to\mathbb R$ satisfies the uniform nondegeneracy
condition
\begin{equation*}
0<m_i\leq \sigma_s^i\leq M_i,
\qquad
s\geq0.
\end{equation*}
Here, $B^H$ denotes an $n$-dimensional fractional Brownian motion with Hurst
parameter $H\in(1/4,1)$. Since $\sigma$ is deterministic and sufficiently
regular in time, the stochastic convolution in \eqref{fir} is well defined as
a Wiener integral and coincides with the corresponding Young integral.

For an admissible reference path $\phi$, the corresponding
Onsager--Machlup functional is defined through the relative small-tube
asymptotics
\begin{equation}
\exp\bigl(-J(\phi)\bigr)
=
\lim_{\varepsilon\downarrow0}
\frac{
\mathbb P
\left(
|X_\cdot-\phi_\cdot|
\leq
\varepsilon
\right)
}{
\mathbb P
\left(
\left|
\int_0^\cdot\sigma_s\,dB_s^H
\right|
\leq
\varepsilon
\right)
},
\end{equation}
provided that the limit exists. Thus, a smaller value of $J(\phi)$ corresponds
to a larger relative probability that the solution path remains in a thin
tube around $\phi$. Most probable paths can therefore be characterized as
minimizers of the Onsager--Machlup functional over an appropriate admissible
path space.

    To investigate the dependence of most probable paths on the noise
intensity, after deriving the Onsager--Machlup functional for
\eqref{fir}, we introduce the small-noise family obtained by replacing
$\sigma$ with $\sqrt{\epsilon}\sigma$. More precisely, we consider
\begin{equation*}
    X_t^\epsilon
    =
    x_0
    +
    \int_0^t b_s(X_s^\epsilon)\,ds
    +
    \sqrt{\epsilon}
    \int_0^t\sigma_s\,dB_s^H,
    \qquad
    \epsilon>0.
\end{equation*}
Here, $\epsilon$ denotes the noise intensity. 

We consider two types of most probable paths. The first is the most probable
transition path, for which both the initial and terminal states are prescribed.
More precisely, given $x_0,x_T\in\mathbb R^n$, we seek minimizers of the
Onsager--Machlup functional among paths satisfying
\begin{equation*}
\phi_0=x_0,
\qquad
\phi_T=x_T.
\end{equation*}
Such paths are particularly relevant to noise-induced transitions between
metastable states. For example, in classical chemical reaction dynamics, the
reactant and product states are commonly represented by two local minima of an
underlying free-energy landscape, and a most probable transition path
describes the most likely reaction route connecting them; see
\cite{pinski2010transition}.

The second object is the most probable evolution path. In this case, only the
initial point is prescribed, while the terminal point is left free. We
therefore seek the most likely evolution of the stochastic system starting
from a given state $x_0$, without imposing a terminal constraint. In a related
setting, \cite{CHENG2019121779} employed most probable transition pathways and
maximal likely trajectories to study concentration transitions and
transcription dynamics in a stochastic genetic regulation model driven by
Gaussian and stable L\'evy noises. The free-endpoint problem enlarges the
admissible path and perturbation spaces, and its relation to the
fixed-endpoint problem will be examined in Section~6.

The Onsager--Machlup functional was originally introduced by Onsager and
Machlup in \cite{Onsager1953}. Subsequently, D\"urr and Bach
\cite{durr1978onsager} established a precise connection among the
Onsager--Machlup function, most probable tubes, and the associated variational
principle. A rigorous mathematical treatment of diffusion processes under the
uniform norm was later developed in \cite{ikeda1981stochastic}, and the theory
was subsequently extended to more general norms on path space in
\cite{Shepp1992,Capitaine1995}.

In recent years, Onsager--Machlup theory has been extended to a variety of
nonclassical stochastic systems. Huang et al.~\cite{duanjq} derived the
Onsager--Machlup functional for jump-diffusion processes using a
probability-flow approach, in which the probability density of the original
jump-diffusion process is related to that of an auxiliary diffusion process.
Liu and Gao studied degenerate stochastic systems driven by fractional
Brownian motion by employing a Girsanov transformation adapted to the
degenerate structure \cite{liu2024onsager,Liu2026}. In a different direction,
\cite{zhang2025persistenceinvarianttoristochastic} combined
Onsager--Machlup theory with KAM techniques to investigate invariant tori in
stochastic nonlinear Schr\"odinger equations. A connection between the
Onsager--Machlup functional and Loewner energy was established in
\cite{Carfagnini2024}. From the perspective of variational convergence,
Pinski et al.~\cite{pinski2012gamma} investigated the $\Gamma$-limit of the
Onsager--Machlup functional in the small-noise regime through a suitable time
rescaling. Building on this work, Li and Li
\cite{doi:10.1137/20M1310539} proved that the $\Gamma$-limit of the
Onsager--Machlup functional on curve space coincides with the geometric
Freidlin--Wentzell functional.

Despite these developments, most existing studies focus primarily on
the derivation of Onsager--Machlup functionals or the computation of
individual most probable paths. Comparatively little is known about the
variational stability of these paths under changes in the noise
intensity, the observation horizon, and the endpoint constraints. This
leads to two closely related questions.

The first question is under what conditions a deterministic trajectory
of the underlying noise-free system remains a local or global minimizer
of the Onsager--Machlup functional after stochastic perturbations are
introduced. We refer to this property as the persistence of most
probable paths. The second question concerns the situation in which
such exact persistence fails: whether global most probable paths still
converge to the corresponding deterministic trajectory in the
small-noise limit and, if so, at what rate and in which path-space
topologies. To the best of our knowledge, such quantitative convergence
results have not been established for multidimensional systems driven
by fractional Brownian motion with time-dependent diffusion
coefficients.

In addition to these fixed-time questions, the long-time behavior of
hyperbolic limit cycles under random perturbations has received
considerable attention. At the pathwise level, Dieci, Li, and Zhou
\cite{Dieci2016RandomPerturbations} showed that, for a suitable class
of random perturbations, stochastic oscillators remain confined near
the underlying deterministic periodic orbit. For small Brownian
perturbations of stable hyperbolic limit cycles, Giacomin, Poquet, and
Shapira \cite{Giacomin2018PhaseDiffusion} established exponentially
long proximity estimates and showed that the phase difference evolves,
on the $\epsilon^{-2}$ time scale, as a Brownian motion with constant
drift. From a statistical viewpoint, Du, Hening, Nguyen, and Yin
\cite{Du2021SwitchingDiffusions} proved the concentration of invariant
measures near stable limit cycles in a rapidly switching diffusion
setting.

These results describe the long-time behavior of periodic structures
through stochastic sample paths or invariant measures. In contrast, we
study the effect of the observation horizon from the variational
viewpoint of most probable-path persistence. More precisely, we ask
whether a deterministic periodic trajectory remains a local minimizer
of the Onsager--Machlup functional when it is periodically extended
over an increasing number of periods.

Let $\bar X$ be a nonconstant $T$-periodic solution of the underlying
noise-free system. Under the fixed-time hypotheses, for every fixed
$N$, there exists a noise threshold $\epsilon_N^\ast>0$ such that the
$N$-fold periodic concatenation of $\bar X$ is a local minimizer of the
corresponding Onsager--Machlup functional whenever
\begin{equation*}
    0<\epsilon<\epsilon_N^\ast.
\end{equation*}
We fix the noise intensity and investigate what happens as the number
of periods $N$ increases. In particular, we ask whether the local
minimality of the periodically extended trajectory can be maintained
uniformly in $N$, or whether a sufficiently long observation horizon
eventually creates a negative direction in the second variation and
destroys most probable-path persistence.
This question is particularly relevant for fractional Brownian motion,
whose temporal correlations are encoded by a nonlocal fractional operator. We therefore further ask whether the memory
carried by the fractional noise can accumulate across successive
periods and thereby alter or destroy the long-time variational
persistence of the deterministic periodic trajectory. As shown below,
the answer depends essentially on the Hurst parameter $H$ through the
long-time scaling of the associated fractional operator.

The main contributions of this paper are summarized as follows.

First, motivated by the genuinely multidimensional structure arising in the
linearized analysis of periodic orbits, we extend the Onsager--Machlup
functional for fractional Brownian motion established in \cite{Nualart} to
multidimensional stochastic differential equations with time-dependent
diffusion coefficients; see
Theorem~\ref{thm:OM-n-dimensional}. The resulting functional is expressed in
terms of a weighted fractional Cameron--Martin operator and is valid for the
full range $H\in(1/4,1)$. Its kinetic part agrees with the corresponding
Freidlin--Wentzell action obtained in \cite{FanYuYuan2023}, while the
one-dimensional correction involving the spatial derivative of the drift is
replaced by the divergence of the drift field, consistently with the
multidimensional Brownian case.

Second, we study the finite-time persistence of deterministic trajectories
as both most probable transition paths with fixed endpoints and most probable
evolution paths with a free terminal point. Let $\bar X$ be a trajectory of
the underlying noise-free system. We impose the criticality condition
\begin{equation*}
    \nabla(\nabla\cdot b_s)(\bar X_s)=0,
\end{equation*}
which means that the local phase-space volume expansion rate
$\nabla\cdot b_s$ has vanishing first spatial variation along $\bar X$.
Under this condition, $\bar X$ is a critical point of the
Onsager--Machlup functional. We establish a Poincar\'e-type
Cameron--Martin embedding estimate and use it to prove the coercivity of
the second variation for sufficiently small noise. This yields the local
persistence result in
Theorem~\ref{thm:local-minimizer-small-noise}. For a sufficiently smaller noise intensity,
Theorem~\ref{thm:global-minimality-small-noise} further shows that
$\bar X$ is the unique global minimizer, and hence that most
probable-path persistence holds globally.

When the criticality condition fails, the deterministic trajectory is
generally no longer an exact minimizer for any fixed $\epsilon>0$.
Nevertheless,
Theorem~\ref{thm:uniform-convergence-global-minimizers},
Corollary~\ref{cor:Holder-convergence-global-minimizers} show that global most probable paths
converge to $\bar X$ in both the uniform and H\"older topologies at the rate
$O(\epsilon)$ as $\epsilon\to0$. Thus, when exact persistence fails, it is
replaced by quantitative small-noise convergence.

Conversely, Theorem~\ref{thm:destruction-large-noise} shows that if the
Hessian of the divergence has a sufficiently negative direction along
$\bar X$, then sufficiently large noise destroys its local minimality in both
endpoint settings. If the divergence is spatially constant, the divergence
correction is independent of the reference path, and every deterministic
trajectory remains a global minimizer for arbitrary noise intensity.

Third, we investigate the long-time behavior of the second variation along a
$T$-periodic deterministic trajectory $\bar X$. For a fixed noise intensity,
we examine the second variation over $[0,NT]$ as $N\to\infty$. The tangential
direction of the periodic orbit generates a neutral mode of the linearized
operator, whose contribution depends crucially on the long-time scaling of
the fractional Cameron--Martin operator.

For $H>1/2$, Theorem~\ref{thm:H-greater-half-collapse} shows that, under an
appropriate effective-potential condition, the destabilizing potential
contribution dominates the kinetic contribution along a suitable family of
slowly varying perturbations. Consequently, the second variation loses
positive definiteness on sufficiently long intervals, and the periodically
extended deterministic trajectory ceases to be a local minimizer.

By contrast, for $1/4<H\leq1/2$,
Theorems~\ref{thm:persistence-H-half} and
\ref{thm:persistence-H-lesshalf} show that, under the corresponding
small-noise conditions, the second variation remains positive definite over
arbitrarily many periods. The resulting estimates are anisotropic: the
tangential component is controlled by the appropriate Cameron--Martin or
fractional kinetic norm, whereas the transversally hyperbolic component is
controlled in $L^2$. These estimates establish long-time anisotropic positive
definiteness, but do not by themselves imply local minimality with respect to
a single path-space topology.

Analytically, the distinction between the two Hurst regimes arises from the
dependence of the long-time scaling of the fractional operator on its order,
which changes across $H=1/2$. Probabilistically, this reflects the
accumulation of perturbations induced by the positive long-range dependence
of fractional Brownian motion when $H>1/2$.

    Fourth, we apply the most probable evolution-path persistence theory to
nearly integrable Hamiltonian systems with time-dependent fractional
stochastic perturbations. Combining the deterministic KAM theorem with our
small-noise persistence result, we show that the resulting deformed invariant
tori persist as families of most probable evolution paths; see
Theorem~\ref{thm:fractional-most-probable-KAM-tori}. This extends the
Brownian-noise result of \cite{xinze2026most} to the fractional setting.

To complement the theoretical analysis, we present two numerical
examples. The first considers a two-dimensional system with an explicit
periodic orbit and verifies the required structural conditions.
Numerical simulations in the Brownian case show that the most probable
transition path remains close to the deterministic periodic orbit under
small noise, whereas a pronounced deviation occurs under large noise.
We further compute the second variation over multiple periods for
different Hurst parameters. The results illustrate the
Hurst-dependent long-time behavior: for $H>1/2$, the second variation
eventually loses positive definiteness, whereas for $H\leq1/2$ it
remains positive over the computed range of periods. The second example
provides a three-dimensional visualization of a deformed KAM torus,
together with an unperturbed quasi-periodic trajectory, a most probable
evolution path on the deformed torus, and sample trajectories under
fractional stochastic perturbations.

Our analysis combines several techniques. In deriving the multidimensional
fractional Onsager--Machlup functional, we use the Gaussian correlation
inequality to establish the conditional expectation limits of the coupling
terms generated by the multidimensional noise. The resulting functional is
naturally defined on a weighted fractional Cameron--Martin space.

To study most probable-path persistence, we compute the first and second
variations of the functional on the corresponding Cameron--Martin spaces.
The vanishing of the first variation gives the criticality condition, while
the coercivity or indefiniteness of the second variation yields persistence
under small noise or its destruction under large noise. To pass from local
to global minimality, we use kinetic-energy sublevel sets
\begin{equation*}
    \left\{
        \phi:I_H(\phi)\leq M
    \right\}.
\end{equation*}
Local minimality controls paths inside a sufficiently small sublevel set,
whereas paths outside it are excluded by the dominant kinetic contribution
for sufficiently small noise. Quantitative convergence to the noise-free
trajectory is obtained from global minimality, the Cameron--Martin embedding,
and Gr\"onwall's inequality.

For periodic trajectories, Floquet theory separates the neutral tangential
mode from the transversally hyperbolic directions. When $H>1/2$, the Fredholm
alternative provides suitable periodic response modes, and the long-time
scaling of the fractional derivative produces destabilizing perturbations.
For $1/4<H\leq1/2$, the Floquet decomposition leads to anisotropic estimates
for the tangential and transverse components. In the fractional regime
$1/4<H<1/2$, we additionally use uniform equivalence between weighted and
unweighted fractional integral norms and uniform multiplier estimates for
periodic matrix coefficients on multiple-period intervals.

The remainder of this paper is organized as follows. Section~2 collects the
necessary preliminaries on fractional Brownian motion, Cameron--Martin spaces,
fractional operators, and the analytic tools used below. In Section~3, we
derive the Onsager--Machlup functional for multidimensional stochastic
differential equations with time-dependent fractional noise; see
Theorem~\ref{thm:OM-n-dimensional}. Section~4 studies most probable paths on
fixed time intervals in both the fixed-endpoint transition problem and the
free-endpoint evolution problem. We establish small-noise persistence,
large-noise loss of local minimality, and quantitative convergence of global
most probable paths to the corresponding deterministic trajectories.
Section~5 is devoted to the long-time behavior of the second variation along
periodic deterministic trajectories. We establish the loss of positive
definiteness for $H>1/2$ and uniform anisotropic positive-definiteness results
for $H\in(1/4,1/2]$; see
Theorems~\ref{thm:H-greater-half-collapse},
\ref{thm:persistence-H-half}, and
\ref{thm:persistence-H-lesshalf}. In Section~6, we establish the persistence
of KAM invariant tori in the sense of most probable evolution paths; see
Theorem~\ref{thm:fractional-most-probable-KAM-tori}. Finally, Section~7 presents two numerical examples illustrating the
theoretical results: a two-dimensional periodic-orbit example and a
three-dimensional visualization of most probable evolution paths
associated with a deformed KAM torus.

\section{Preliminaries}  \label{perliminary}
In this section, we recall some basic notations, assumptions and lemmas that will be used in the sequel.

Throughout this paper, for $1\leq p\leq \infty$, we denote by
\[
    L^p([a,b];\mathbb R^n)
\]
the usual Lebesgue space of $\mathbb R^n$-valued measurable functions on
$[a,b]$, equipped with the norm $\|\cdot\|_{L^p}$. We write
\[
    C([a,b];\mathbb R^n)
\]
for the space of continuous functions from $[a,b]$ to $\mathbb R^n$.

For $\alpha\in(0,1)$, we denote by
\[
    C^\alpha([a,b];\mathbb R^n)
\]
the Hölder space consisting of all functions $f\in C([a,b];\mathbb R^n)$ such
that
\[
    [f]_\alpha
    :=
    \sup_{\substack{s,t\in[a,b]\\ s\neq t}}
    \frac{|f_t-f_s|}{|t-s|^\alpha}
    <\infty .
\]
It is equipped with the norm
\[
    \|f\|_{C^\alpha}
    :=
    \|f\|_\infty+[f]_\alpha .
\]
When the initial value vanishes, we write
\[
    C_0^\alpha([a,b];\mathbb R^n)
    :=
    \left\{
    f\in C^\alpha([a,b];\mathbb R^n): f_a=0
    \right\}.
\]
On this space, the Hölder seminorm $[\cdot]_\alpha$ is equivalent to the full
Hölder norm. Hence we endow $C_0^\alpha([a,b];\mathbb R^n)$ with the norm
\[
    \|f\|_\alpha := [f]_\alpha .
\]

\subsection{Fractional calculus}

We first introduce the basic concepts of fractional calculus; for further details, refer to \cite{Samko1993}.
 \begin{definition}
   	Let $f\in L^1([a,b];\mathbb{R}^n)$. The integrals 
   	\begin{align*}
   		(I_{a^+}^\alpha f)(x)&:=\frac{1}{\Gamma(\alpha)}\int_a^x (x-y)^{\alpha-1}f(y)dy,\quad x\geq a,\\
   		(I_{b^-}^\alpha f)(x)&:=\frac{1}{\Gamma(\alpha)}\int_x^b (y-x)^{\alpha-1}f(y)dy,\quad x\leq b,
   	\end{align*}
   	where $\alpha>0$, are respectively called the right and left Riemann--Liouville fractional integrals of order $\alpha$.

   \end{definition}

   	\begin{definition}
   		 Let $f\in I_{a^+}^\alpha(L^p) $, $g\in I_{b^-}^\alpha(L^p) $.
   		 Each of the expressions 
   		 \begin{align*}
   		(D_{a^+}^\alpha f)(x)&:=\left(\frac{d}{dx} \right)^{[\alpha]+1}I_{a^+}^{1+[\alpha]-\alpha}f(x),\\
   		(D_{b^-}^\alpha g)(x)&:=\left(-\frac{d}{dx} \right)^{[\alpha]+1}I_{b^-}^{1+[\alpha]-\alpha}g(x)
   	   	\end{align*}
   		 are respectively called the left and right fractional derivative.
   \end{definition}

	If $f\in I_{a^+}^\alpha(L^p) $, then the function $ \phi $ such that $f=I_{a^+}^\alpha(\phi)$ is unique in $L^p $. Fractional derivatives can be regarded as the inverse operation of fractional integrals.
 
 When $\alpha p > 1$, any function in $I_{a^+}^\alpha(L^p)$ is 
$(\alpha - \tfrac{1}{p})$-H\"older continuous. Moreover, every H\"older continuous 
function of order $\beta > \alpha$ admits a fractional derivative of order 
$\alpha$; see \cite[Proposition~2.1]{Decreusefond1999}.

  	\subsection{The \texorpdfstring{$\sigma$}{sigma}-weighted fractional
Cameron--Martin space}

Let
\begin{equation*}
    \Omega:=C_0([0,T];\mathbb R^n),
    \qquad
    \mathcal F:=\mathcal B(\Omega),
\end{equation*}
and let $B_t(\omega):=\omega(t)$ be the canonical process. For
$H\in(0,1)$, denote by $\mathbb P_H$ the centered Gaussian measure under
which $B$, denoted by $B^H$, is an $n$-dimensional fractional Brownian
motion with independent components and covariance
\begin{equation*}
    \mathbb E_{\mathbb P_H}
    \left[
        B_t^{H,i}B_s^{H,j}
    \right]
    =
    \delta_{ij}R_H(t,s),
    \qquad
    R_H(t,s)
    =
    \frac12
    \left(
        t^{2H}+s^{2H}-|t-s|^{2H}
    \right).
\end{equation*}
For every $\varepsilon\in(0,H)$,
\begin{equation*}
    B^H\in
    C_0^{H-\varepsilon}([0,T];\mathbb R^n),
    \qquad
    \mathbb P_H\text{-almost surely}.
\end{equation*}
Throughout the paper, we set
\begin{equation*}
    \alpha:=\left|H-\frac12\right|.
\end{equation*}

For the fractional Girsanov transformation, we use the Volterra
representation
\begin{equation*}
    B_t^H
    =
    \int_0^t K_H(t,s)\,dW_s,
\end{equation*}
where $W$ is an $n$-dimensional Brownian motion on an underlying Wiener
space and $K_H$ is the standard fractional Brownian kernel. The associated
deterministic operator is
\begin{equation}
\label{kh_operator}
    (K_Hu)(t)
    :=
    \int_0^t K_H(t,s)u(s)\,ds,
    \qquad
    u\in L^2([0,T];\mathbb R^n).
\end{equation}

The Cameron--Martin space of $B^H$ is
\begin{equation*}
    \mathcal H_H([0,T];\mathbb R^n)
    :=
    K_H\bigl(L^2([0,T];\mathbb R^n)\bigr),
\end{equation*}
equipped with
\begin{equation*}
    \langle K_Hu,K_Hv\rangle_{\mathcal H_H}
    :=
    \langle u,v\rangle_{L^2}.
\end{equation*}
In particular,
\begin{equation*}
    \mathcal H_{1/2}([0,T];\mathbb R^n)
    =
    \left\{
        h\in AC([0,T];\mathbb R^n):
        h(0)=0,\;
        h'\in L^2(0,T;\mathbb R^n)
    \right\}.
\end{equation*}

The inverse of $K_H$ admits the following fractional-calculus
representation; see \cite[Lemma~10]{Nualart}:
\begin{equation}
\label{eq:KH_inverse}
    K_H^{-1}h(s)
    =
    \begin{cases}
        \displaystyle
        s^\alpha D_{0^+}^\alpha
        \left[
            s^{-\alpha}D_{0^+}^{1-2\alpha}h
        \right](s),
        & H<\dfrac12,
        \\[2.5ex]
        \displaystyle
        h'(s),
        & H=\dfrac12,
        \\[2.5ex]
        \displaystyle
        s^\alpha D_{0^+}^\alpha
        \left[
            s^{-\alpha}h'
        \right](s),
        & H>\dfrac12.
    \end{cases}
\end{equation}
For sufficiently regular $h$ and $H<1/2$, the first expression reduces to
\begin{equation}
\label{eq:KH_inverse_H_less_half_smooth}
    K_H^{-1}h(s)
    =
    s^{-\alpha}
    I_{0^+}^\alpha
    \left[
        s^\alpha h'
    \right](s),
    \qquad
    \alpha=\frac12-H.
\end{equation}

We shall also use the following standard Young integration estimate.

\begin{lemma}[Young integration]
\label{lem:young-integration}
Let $f\in C^\beta([0,T])$ and $g\in C^\gamma([0,T])$, where
$\beta,\gamma\in(0,1)$ and $\beta+\gamma>1$. Then the Young integral
$\int_0^t f_s\,dg_s$ is well defined and
\begin{equation}
\label{eq:young-loeve-estimate}
    \left|
        \int_s^t f_r\,dg_r
        -
        f_s(g_t-g_s)
    \right|
    \leq
    C_{\beta,\gamma}
    [f]_\beta[g]_\gamma
    |t-s|^{\beta+\gamma}.
\end{equation}
The result extends componentwise to vector- and matrix-valued functions.
\end{lemma}

We impose the following standing assumptions.

\begin{assumption}
\label{ass:A}
The drift satisfies
\begin{equation*}
    b\in
    C_b^{1,3}
    \bigl([0,T]\times\mathbb R^n;\mathbb R^n\bigr).
\end{equation*}
The diffusion matrix
\begin{equation*}
    \sigma_t
    =
    \operatorname{diag}
    \bigl(\sigma_t^1,\ldots,\sigma_t^n\bigr)
\end{equation*}
is deterministic, with $\sigma^i\in C^1([0,T])$, and there exist
$m_i,M_i>0$ such that
\begin{equation*}
    0<m_i\leq\sigma_t^i\leq M_i,
    \qquad
    t\in[0,T],
    \quad
    1\leq i\leq n.
\end{equation*}
\end{assumption}

Define the modified Volterra operator by
\begin{equation}
\label{khsigma}
    (K_H^\sigma u)(t)
    :=
    \int_0^t \sigma_s\,d(K_Hu)(s),
    \qquad
    u\in L^2([0,T];\mathbb R^n),
\end{equation}
where the integral is understood in the Young sense. The
$\sigma$-weighted fractional Cameron--Martin space is
\begin{equation}
\label{eq:sigma_weighted_CM_space}
    \mathcal H_H^\sigma([0,T];\mathbb R^n)
    :=
    K_H^\sigma
    \bigl(L^2([0,T];\mathbb R^n)\bigr),
\end{equation}
with inner product
\begin{equation*}
    \langle h,g\rangle_{\mathcal H_H^\sigma}
    :=
    \left\langle
        (K_H^\sigma)^{-1}h,
        (K_H^\sigma)^{-1}g
    \right\rangle_{L^2}.
\end{equation*}
Equivalently,
\begin{equation*}
    \mathcal H_H^\sigma
    =
    \left\{
        h\in C_0([0,T];\mathbb R^n):
        \int_0^\cdot\sigma_s^{-1}\,dh_s
        \in\mathcal H_H
    \right\},
\end{equation*}
and
\begin{equation}
\label{eq:Ksigma_inverse}
    (K_H^\sigma)^{-1}h
    =
    K_H^{-1}
    \left(
        \int_0^\cdot\sigma_s^{-1}\,dh_s
    \right).
\end{equation}

For sufficiently smooth $h$,
\begin{equation}
\label{eq:Ksigma_inverse_explicit}
    (K_H^\sigma)^{-1}h(s)
    =
    \begin{cases}
        \displaystyle
        s^{-\alpha}I_{0^+}^\alpha
        \left(
            s^\alpha\sigma_s^{-1}h'(s)
        \right),
        & H<\dfrac12,
        \quad
        \alpha=\dfrac12-H,
        \\[2.5ex]
        \displaystyle
        \sigma_s^{-1}h'(s),
        & H=\dfrac12,
        \\[2.5ex]
        \displaystyle
        s^\alpha D_{0^+}^\alpha
        \left(
            s^{-\alpha}\sigma_s^{-1}h'(s)
        \right),
        & H>\dfrac12,
        \quad
        \alpha=H-\dfrac12.
    \end{cases}
\end{equation}

Since $\sigma$ and $\sigma^{-1}$ are bounded and continuously
differentiable, the transformation
\begin{equation*}
    h\longmapsto\int_0^\cdot\sigma_s\,dh_s
\end{equation*}
is an isomorphism of $\mathcal H_H$. Consequently,
$\mathcal H_H^\sigma$ and $\mathcal H_H$ coincide as sets and have
equivalent norms. In particular, for every $\beta<H$,
\begin{equation*}
    \mathcal H_H^\sigma([0,T];\mathbb R^n)
    \hookrightarrow
    C_0^\beta([0,T];\mathbb R^n)
\end{equation*}
continuously.

	        \subsection{Small-ball estimates and the GCI }
In this subsection, we use the Gaussian correlation inequality (GCI) to
control the coupling terms arising from multidimensional fractional noise.

\begin{theorem}[Gaussian Correlation Inequality]
Let $\mu$ be a centered Radon Gaussian measure on a separable Banach space $E$. For any two $\mu$-measurable, symmetric, and convex sets $C_1, C_2 \subseteq E$, the following inequality holds:
\begin{equation} \label{GCI}
\mu(C_1 \cap C_2) \geq \mu(C_1) \mu(C_2).
\end{equation}
\end{theorem}

\begin{remark}
While the Gaussian Correlation Inequality \eqref{GCI} was originally proved by Royen \cite{royen2014convex} (and further clarified by Lata{\l}a and Matlak \cite{latala2017royen}) for standard Gaussian measures on finite-dimensional Euclidean spaces $\mathbb{R}^d$, its extension to any separable Banach space $E$ equipped with a centered Radon Gaussian measure $\mu$ is a standard consequence of measure-theoretic approximation.

Specifically, by the  Hahn-Banach theorem, any closed, symmetric, and convex set $C \subseteq E$ is weakly closed and can be represented as the countable intersection of symmetric strips:
\[
    C = \bigcap_{i=1}^\infty \left\{ \omega \in E : |l_i(\omega)| \leq c_i \right\},
\]
where $\{l_i\}_{i=1}^\infty \subset E^*$ is a countable separating family of continuous linear functionals and $c_i > 0$. For any integer $n \geq 1$, define the finite-dimensional cylindrical sets
\[
    C^{(n)}_1 = \bigcap_{i=1}^n \left\{ \omega \in E : |l^{(1)}_i(\omega)| \leq c^{(1)}_i \right\} \quad \text{and} \quad C^{(n)}_2 = \bigcap_{i=1}^n \left\{ \omega \in E : |l^{(2)}_i(\omega)| \leq c^{(2)}_i \right\}.
\]
Since $C^{(n)}_1$ and $C^{(n)}_2$ depend only on the finite-dimensional projections of the Gaussian measure $\mu$, the finite-dimensional GCI is directly applicable, yielding $\mu\big(C^{(n)}_1 \cap C^{(n)}_2\big) \geq \mu\big(C^{(n)}_1\big) \mu\big(C^{(n)}_2\big)$.
Because the sequences of sets $\big\{C^{(n)}_1\big\}_{n=1}^\infty$ and $\big\{C^{(n)}_2\big\}_{n=1}^\infty$ are monotonically decreasing with $C_k = \bigcap_{n=1}^\infty C^{(n)}_k$ for $k=1,2$, the continuity of the Radon measure $\mu$ from above guarantees that the inequality passes to the limit as $n \to \infty$. The result for arbitrary $\mu$-measurable symmetric convex sets then follows by inner regularity.
\end{remark}

\begin{remark}
In our context, the $\varepsilon$-tube in the H\"{o}lder space, defined by $A_\varepsilon = \{ \omega : \|\int_0^\cdot \sigma_u dB^H_u(\omega)\|_\beta < \varepsilon \}$, constitutes a symmetric convex set. 
\end{remark}

\begin{corollary} \label{gauss set}
Let $(\Omega, \mathcal{F}, \mathbb{P})$ be an abstract Wiener space (where $\Omega$ is a separable Banach space and $\mathbb{P}$ is the centered standard Gaussian measure on $\Omega$). 
Let $A \subseteq \Omega$ be a $\mu$-measurable, symmetric, and convex set, and let $f: \Omega \to [0, \infty)$ be a $\mathbb{P}$-measurable, symmetric, and convex function. Then for every $t > 0$, the following estimate holds:
\begin{equation*}
    \mathbb{P}\bigl(\{x \in \Omega : f(x) \le t\} \cap A\bigr)
    \ge
    \mathbb{P}\bigl(\{x \in \Omega: f(x) \le t\}\bigr) \mathbb{P}(A).
\end{equation*}
\end{corollary}

\begin{proof}
By the assumptions on $f$, the lower level set $C_t := \{x \in \Omega : f(x) \le t\}$ is $\mu$-measurable, symmetric, and convex in $\Omega$ for any given $t > 0$. Since $\Omega$ is a separable Banach space and $\mathbb{P}$ is a centered Gaussian measure on $\Omega$, we can directly apply the Gaussian Correlation Inequality \eqref{GCI} to the sets $C_t$ and $A$. This immediately yields
\begin{equation*}
    \mathbb{P}(C_t \cap A) \ge \mathbb{P}(C_t) \mathbb{P}(A),
    \end{equation*}
which completes the proof.
\end{proof}

\begin{remark}
It is worth noting that in the absence of a general proof for \eqref{GCI} at the time, \cite[Corollary 4.6.3]{Bogachev1998} established the above inequality only for the specific case where the level sets are defined by linear functionals. The breakthrough by Royen \cite{royen2014convex} allows us to extend this result to the general convex setting presented here.
\end{remark}

\begin{lemma}[Layer Cake Representation]\label{tail}
Let $Y$ be a non-negative random variable. Then
\begin{equation*}
    \mathbb{E}[Y] = \int_0^\infty \mathbb{P}(Y > s) \, ds.
\end{equation*}
\end{lemma}

The following theorem demonstrates that under the symmetry and convexity assumptions, the conditional expectation over a symmetric convex set is dominated by the unconditional expectation.

\begin{lemma}\label{conditional-expectation}
Let $f$ be a non-negative, convex, symmetric, and measurable function on the Wiener space $(\Omega, \mathcal{F}, \mathbb{P})$. Let $A \in \mathcal{F}$ be a symmetric convex set with $\mathbb{P}(A) > 0$. Then the following inequality holds:
\begin{equation*}
    \mathbb{E}\bigl[f \mid A\bigr] \le \mathbb{E}\bigl[f\bigr].
\end{equation*}
\end{lemma}

\begin{proof}
For each $t \ge 0$, we consider the level set $S_t := \{\omega \in \Omega : f(\omega) \le t\}$. Since $f$ is assumed to be symmetric and convex, $S_t$ is a symmetric convex measurable set in $\Omega$. Applying Corollary \ref{gauss set} (with $\mu = \mathbb{P}$ and $X = \Omega$), we have
\begin{equation*}
    \mathbb{P}(S_t \cap A) \ge \mathbb{P}(S_t) \mathbb{P}(A).
\end{equation*}
By the properties of the probability measure restricted to the set $A$, we observe that
\begin{align*}
    \mathbb{P}(S_t^c \cap A) &= \mathbb{P}(A) - \mathbb{P}(S_t \cap A) \\
    &\le \mathbb{P}(A) - \mathbb{P}(S_t) \mathbb{P}(A) \\
    &= (1 - \mathbb{P}(S_t)) \mathbb{P}(A) = \mathbb{P}(S_t^c) \mathbb{P}(A).
\end{align*}
Dividing both sides by $\mathbb{P}(A)$ yields the comparison for the conditional tail probabilities:
\begin{equation*}
    \mathbb{P}(f > t \mid A) = \frac{\mathbb{P}(S_t^c \cap A)}{\mathbb{P}(A)} \le \mathbb{P}(S_t^c) = \mathbb{P}(f > t).
\end{equation*}
Integrating the above tail probability with respect to $t$ and invoking Lemma \ref{tail}, we conclude that
\begin{align*}
    \mathbb{E}\bigl[f \mid A \bigr] &= \int_0^\infty \mathbb{P}(f > t \mid A) \, dt \\
    &\le \int_0^\infty \mathbb{P}(f > t) \, dt \\
    &= \mathbb{E}\bigl[f\bigr].
\end{align*}
This completes the proof.
\end{proof}

\subsection{Operator-theoretic preliminaries}

We recall several operator-theoretic results used below; see
\cite{brezis2011functional,evans2010partial}. For Banach spaces $X$ and $Y$,
let $\mathcal B(X,Y)$ denote the space of bounded linear operators from $X$
to $Y$, and write $\mathcal B(X):=\mathcal B(X,X)$.

\begin{theorem}[Fredholm alternative]
\label{thm:fredholm_alternative}
Let $X$ be a Banach space, $A\in\mathcal B(X)$ be compact, and
$\lambda\neq0$. Then exactly one of the following alternatives holds:
\begin{enumerate}
    \item
    $\ker(\lambda I-A)=\{0\}$, in which case $\lambda I-A$ is boundedly
    invertible on $X$;
    \item
    $\ker(\lambda I-A)\neq\{0\}$, in which case
    \begin{equation*}
        \dim\ker(\lambda I-A)
        =
        \dim\ker(\lambda I^*-A^*)
        <\infty,
    \end{equation*}
    and the equation
    \begin{equation*}
        (\lambda I-A)u=f
    \end{equation*}
    is solvable if and only if
    \begin{equation*}
        \langle f,v^*\rangle=0
        \qquad
        \text{for every }
        v^*\in\ker(\lambda I^*-A^*).
    \end{equation*}
\end{enumerate}
\end{theorem}

A Volterra integral operator on $[0,T]$ has the form
\begin{equation*}
    (Kf)(t)
    =
    \int_0^t k(t,s)f(s)\,ds.
\end{equation*}
The following standard result will be used repeatedly; see
\cite[Sec.~9, Cor.~3.18]{gripenberg1990volterra}.

\begin{theorem}[Weakly singular Volterra operators]
\label{volterra}
Let $1<p<\infty$, and suppose that the measurable Volterra kernel
$k:(0,T)^2\to\mathbb R^{d\times d}$ satisfies
\begin{equation*}
    |k(t,s)|
    \leq
    C(t-s)^{-\alpha},
    \qquad
    0<s<t<T,
\end{equation*}
for some $\alpha\in(0,1)$. Then the operator
\begin{equation*}
    (Kf)(t)
    :=
    \int_0^t k(t,s)f(s)\,ds
\end{equation*}
is bounded and quasinilpotent on $L^p(0,T;\mathbb R^d)$:
\begin{equation*}
    \lim_{m\to\infty}
    \|K^m\|_{\mathcal B(L^p)}^{1/m}
    =
    0.
\end{equation*}
Consequently, $I-\lambda K$ is boundedly invertible for every
$\lambda\in\mathbb C$.
\end{theorem}

The Riemann--Liouville fractional integral
\begin{equation*}
    (I_{0+}^{\alpha}f)(t)
    =
    \frac{1}{\Gamma(\alpha)}
    \int_0^t
    (t-s)^{\alpha-1}f(s)\,ds,
    \qquad
    0<\alpha<1,
\end{equation*}
is a typical weakly singular Volterra operator.

Finally, let $X$ be a real Banach space and
$a:X\times X\to\mathbb R$ be a bilinear form. We call $a$ positive
definite if
\begin{equation*}
    a(x,x)>0,
    \qquad
    x\neq0,
\end{equation*}
and coercive if there exists $c>0$ such that
\begin{equation*}
    a(x,x)
    \geq
    c\|x\|_X^2,
    \qquad
    x\in X.
\end{equation*}

\subsection{Floquet coordinates and the linearized operator}
\label{subsec:floquet-linearized-operator}

We recall the Floquet representation used in the long-time analysis; see
\cite{perko2001differential}. Let $\bar X$ be a nonconstant $T$-periodic
solution of
\begin{equation*}
    \dot X_t=b(X_t).
\end{equation*}
The variational equation along $\bar X$ is
\begin{equation}
\label{eq:variational-equation}
    \dot v_t=\nabla b(\bar X_t)v_t.
\end{equation}
If $\Phi(t,0)$ denotes its principal fundamental matrix, Floquet theory gives
\begin{equation}
\label{eq:floquet-decomposition}
    \Phi(t,0)=P(t)e^{tR},
    \qquad
    P(t+T)=P(t),
\end{equation}
where $P(t)$ is invertible and $R$ is constant.

Since the system is autonomous, $\dot{\bar X}_t$ solves
\eqref{eq:variational-equation}. Hence $1$ is always a Floquet multiplier,
corresponding to the tangential direction of the periodic orbit.

\begin{definition}[Transversally hyperbolic periodic orbit]
\label{def:transversally-hyperbolic}
The periodic orbit $\bar X$ is transversally hyperbolic if the Floquet
multiplier $1$ is simple and every other Floquet multiplier $\mu$ satisfies
\begin{equation*}
    |\mu|\neq1.
\end{equation*}
\end{definition}

In Floquet coordinates adapted to the tangential and transverse directions,
we may write
\begin{equation}
\label{eq:R-block}
    R=
    \begin{pmatrix}
        0 & 0\\
        0 & R_\perp
    \end{pmatrix},
\end{equation}
where every eigenvalue of $R_\perp$ has nonzero real part. The first column
of $P(t)$ may be chosen as the tangential mode $\dot{\bar X}_t$.

Define the linearized operator along $\bar X$ by
\begin{equation}
\label{eq:linearized-operator}
    \mathcal L\eta_t
    :=
    \dot\eta_t-\nabla b(\bar X_t)\eta_t.
\end{equation}
Writing
\begin{equation}
\label{eq:floquet-coordinate}
    \eta_t=P(t)z_t,
\end{equation}
and differentiating \eqref{eq:floquet-decomposition}, we obtain
\begin{equation}
\label{eq:L-floquet-coordinate}
    \mathcal L\eta_t
    =
    P(t)\bigl(\dot z_t-Rz_t\bigr).
\end{equation}
Since $P$ is $T$-periodic, this representation holds on every interval
$[0,NT]$, $N\in\mathbb N^+$.

\section{Onsager--Machlup Functionals for $n$-Dimensional Time-Varying Fractional Noise}

In this section, we derive the Onsager--Machlup functional associated with
\eqref{fir}. The proof is divided into two parts. In the first part, we use a
Girsanov transformation to rewrite the probability ratio
\begin{equation*}
\frac{
\mathbb P\bigl(\|X_\cdot-\phi_\cdot\|\leq \varepsilon\bigr)
}{
\mathbb P\bigl(
\|\int_0^\cdot \sigma_s\,dB_s^H\|\leq \varepsilon
\bigr)
}
\end{equation*}
as a conditional expectation. In the second part, we estimate this conditional
expectation and complete the derivation of the Onsager--Machlup functional.

\subsection{Derivation of the Onsager--Machlup functional}

In this subsection, we derive the Onsager--Machlup functional for
\(1/4<H<1\).

\begin{theorem}
\label{thm:OM-n-dimensional}
Let \(X\) be the solution of \eqref{fir}, and assume that the coefficients
satisfy Assumption~\ref{ass:A}. For any admissible path \(\phi\) such that
\[
    \phi\in x_0+\mathcal H_H^\sigma\bigl([0,T];\mathbb R^n\bigr),
\]
we denote by \(\dot\phi\) the element of \(L^2([0,T];\mathbb R^n)\) satisfying $
    \phi-x_0=K_H^\sigma(\dot\phi).
$
Then, for \(1/4<H<1\), the Onsager--Machlup functional of \(X\) with respect
to the H\"older norm \(\|\cdot\|_\beta\), where
\[
    \max\{H-1/2,0\}<\beta<H-1/4,
\]
is given by
\begin{equation}
\label{eq:OM-functional-n-dimensional}
\begin{aligned}
J(\phi)
&=
\frac12\int_0^T
\left|
(K_H^\sigma)^{-1}
\left(
    \phi_\cdot-x_0-\int_0^\cdot b_v(\phi_v)\,dv
\right)(s)
\right|^2\,ds  \\
&\quad
+\frac{d_H}{2}\int_0^T \nabla\cdot b_s(\phi_s)\,ds ,
\end{aligned}
\end{equation}
where
\[
    d_H
    =
    \sqrt{
    \frac{
    2H\,\Gamma\!\left(H+\frac12\right)
    \Gamma\!\left(\frac32-H\right)
    }{
    \Gamma(2-2H)
    }} .
\]
\end{theorem}

\begin{proof}
The one-dimensional part of the argument follows the same strategy as in
\cite{Nualart}. Therefore, we only present the part of the proof which is
specific to the multidimensional setting, namely the treatment of the coupling
terms arising from the Taylor expansion of the stochastic integral.

We first consider the case \(1/4<H<1/2\). 
The key term to be evaluated is the conditional expectation
\begin{equation*}
\mathbb E
\left[
\exp\left\{
\int_0^T
s^{-\alpha}I_{0^+}^{\alpha}s^\alpha
\Bigl[
    \sigma_s^{-1}
    b_s\Bigl(
        \phi_s+\int_0^s\sigma_u\,dB_u^H
    \Bigr)
\Bigr]
\,dW_s
\right\}
\,\middle|\,
\bigcap_{i=1}^n A_i
\right],
\end{equation*}
where
\[
    A_i
    :=
    \left\{
    \left\|
    \int_0^\cdot \sigma_u^i\,dB_u^{H,i}
    \right\|_\beta
    \leq \varepsilon
    \right\},
    \qquad i=1,\dots,n .
\]
Here, for notational simplicity, we write \(\sigma^i\) for the \(i\)-th
diagonal component of \(\sigma\). In the non-diagonal case, the same argument
applies after working componentwise with the corresponding independent
Gaussian coordinates.

Expanding the drift term around \(\phi\), the first-order part is
\begin{equation*}
\sum_{i,j=1}^n
\int_0^T
s^{-\alpha}I_{0^+}^{\alpha}s^\alpha
\left[
    (\sigma_s^i)^{-1}
    \partial_j b_s^i(\phi_s)
    \left(
        \int_0^s \sigma_u^j\,dB_u^{H,j}
    \right)
\right]
\,dW_s^i .
\end{equation*}
We distinguish between the diagonal terms \(i=j\) and the off-diagonal terms
\(i\neq j\).

For \(i=j\), by the independence of different components, the conditional
expectation with respect to \(\bigcap_{k=1}^n A_k\) reduces to the conditional
expectation with respect to \(A_i\). Hence
\begin{align*}
&\mathbb E
\left[
\exp\left\{
\int_0^T
s^{-\alpha}I_{0^+}^{\alpha}s^\alpha
\left[
    (\sigma_s^i)^{-1}
    \partial_i b_s^i(\phi_s)
    \left(
        \int_0^s \sigma_u^i\,dB_u^{H,i}
    \right)
\right]
\,dW_s^i
\right\}
\,\middle|\,
\bigcap_{k=1}^n A_k
\right]
\nonumber \\
&\quad =
\mathbb E
\left[
\exp\left\{
\int_0^T
s^{-\alpha}I_{0^+}^{\alpha}s^\alpha
\left[
    (\sigma_s^i)^{-1}
    \partial_i b_s^i(\phi_s)
    \left(
        \int_0^s \sigma_u^i\,dB_u^{H,i}
    \right)
\right]
\,dW_s^i
\right\}
\,\middle|\,
A_i
\right].
\end{align*}
By the trace estimate in \cite{Nualart}, we obtain
\begin{align*}
&\lim_{\varepsilon\to0}
\mathbb E
\left[
\exp\left\{
\int_0^T
s^{-\alpha}I_{0^+}^{\alpha}s^\alpha
\left[
    (\sigma_s^i)^{-1}
    \partial_i b_s^i(\phi_s)
    \left(
        \int_0^s \sigma_u^i\,dB_u^{H,i}
    \right)
\right]
\,dW_s^i
\right\}
\,\middle|\,
A_i
\right]
\nonumber \\
&\quad =
\exp\left\{
-\frac{d_H}{2}
\int_0^T \partial_i b_s^i(\phi_s)\,ds
\right\}.
\end{align*}

It remains to show that the off-diagonal terms do not contribute to the
limiting Onsager--Machlup functional. Fix \(i\neq j\). Then the corresponding
conditional expectation is
\begin{align*}
&\mathbb E
\left[
\exp\left\{
\int_0^T
s^{-\alpha}I_{0^+}^{\alpha}s^\alpha
\left[
    (\sigma_s^i)^{-1}
    \partial_j b_s^i(\phi_s)
    \left(
        \int_0^s \sigma_u^j\,dB_u^{H,j}
    \right)
\right]
\,dW_s^i
\right\}
\,\middle|\,
\bigcap_{k=1}^n A_k
\right]
\nonumber \\
&\quad =
\mathbb E
\left[
\exp\left\{
\int_0^T
s^{-\alpha}I_{0^+}^{\alpha}s^\alpha
\left[
    (\sigma_s^i)^{-1}
    \partial_j b_s^i(\phi_s)
    \left(
        \int_0^s \sigma_u^j\,dB_u^{H,j}
    \right)
\right]
\,dW_s^i
\right\}
\,\middle|\,
A_i,A_j
\right].
\end{align*}
For fixed \(\omega_j\), define
\[
    h_s(\omega_j)
    :=
    s^{-\alpha}I_{0^+}^{\alpha}s^\alpha
    \left[
        (\sigma_s^i)^{-1}
        \partial_j b_s^i(\phi_s)
        \left(
            \int_0^s \sigma_u^j\,dB_u^{H,j}(\omega_j)
        \right)
    \right].
\]
Then, using Fubini's theorem and the independence of the \(i\)-th and
\(j\)-th components, we can write
\begin{align*}
&\mathbb E
\left[
\exp\left\{
    \int_0^T h_s(\omega_j)\,dW_s^i
\right\}
\,\middle|\,
A_i,A_j
\right]
\nonumber \\
&\quad =
\mathbb E_j
\left[
\mathbb E_i
\left(
\exp\left\{
    \int_0^T h_s(\omega_j)\,dW_s^i(\omega_i)
\right\}
\,\middle|\,
A_i
\right)
\,\middle|\,
A_j
\right].
\end{align*}

We now estimate the inner conditional expectation uniformly for
\(\omega_j\in A_j\). Since
\[
    \int_0^T h_s(\omega_j)\,dW_s^i
\]
is a centered Gaussian random variable with respect to the \(i\)-th Gaussian
coordinate, and \(A_i\) is symmetric and convex, Jensen's inequality gives
\begin{align*}
\mathbb E_i
\left(
\exp\left\{
    \int_0^T h_s(\omega_j)\,dW_s^i
\right\}
\,\middle|\,
A_i
\right)
&\geq
\exp\left\{
\mathbb E_i
\left(
    \int_0^T h_s(\omega_j)\,dW_s^i
    \,\middle|\,
    A_i
\right)
\right\}
\nonumber \\
&=1 .
\end{align*}
For the upper bound, we apply Lemma~\ref{conditional-expectation} to obtain
\begin{align*}
\mathbb E_i
\left(
\exp\left\{
    \int_0^T h_s(\omega_j)\,dW_s^i
\right\}
\,\middle|\,
A_i
\right)
&\leq
\mathbb E_i
\exp\left\{
    \int_0^T h_s(\omega_j)\,dW_s^i
\right\}
\nonumber \\
&=
\exp\left\{
    \frac12\int_0^T |h_s(\omega_j)|^2\,ds
\right\}.
\end{align*}
On the event \(A_j\), the small-ball condition implies
\[
    \left\|
    \int_0^\cdot \sigma_u^j\,dB_u^{H,j}
    \right\|_\beta
    \leq \varepsilon .
\]
By the boundedness of \((\sigma^i)^{-1}\partial_j b^i\) and the mapping
property of \(s^{-\alpha}I_{0^+}^{\alpha}s^\alpha\), there exists a constant
\(C>0\), independent of \(\varepsilon\), such that
\[
    \|h(\omega_j)\|_{L^2([0,T])}
    \leq C\varepsilon ,
    \qquad \omega_j\in A_j .
\]
Therefore,
\[
    1
    \leq
    \mathbb E_i
    \left(
    \exp\left\{
        \int_0^T h_s(\omega_j)\,dW_s^i
    \right\}
    \,\middle|\,
    A_i
    \right)
    \leq
    \exp\left\{
        \frac{C^2\varepsilon^2}{2}
    \right\},
    \qquad \omega_j\in A_j .
\]
Taking the outer conditional expectation with respect to \(A_j\) preserves
these bounds. Hence, by the squeeze theorem,
\[
\lim_{\varepsilon\to0}
\mathbb E
\left[
\exp\left\{
    \int_0^T h_s(\omega_j)\,dW_s^i
\right\}
\,\middle|\,
A_i,A_j
\right]
=
1 .
\]
Thus all off-diagonal terms vanish in the limit.

Combining the diagonal limits over \(i=1,\dots,n\), we obtain
\[
    \prod_{i=1}^n
    \exp\left\{
    -\frac{d_H}{2}
    \int_0^T \partial_i b_s^i(\phi_s)\,ds
    \right\}
    =
    \exp\left\{
    -\frac{d_H}{2}
    \int_0^T \nabla\cdot b_s(\phi_s)\,ds
    \right\}.
\]
Together with the deterministic Cameron--Martin contribution coming from the
Girsanov density, this yields
\[
\begin{aligned}
J(\phi)
&=
\frac12\int_0^T
\left|
(K_H^\sigma)^{-1}
\left(
    \phi_\cdot-x_0-\int_0^\cdot b_v(\phi_v)\,dv
\right)(s)
\right|^2\,ds  \\
&\quad
+\frac{d_H}{2}\int_0^T \nabla\cdot b_s(\phi_s)\,ds .
\end{aligned}
\]

The case \(1/2\leq H<1\) is handled in the same way. The only difference lies in the
fractional operator appearing in the representation of \((K_H^\sigma)^{-1}\).
The diagonal terms again yield the trace contribution, whereas the
off-diagonal terms vanish by Lemma~\ref{conditional-expectation}. This
completes the proof.
\end{proof}

\begin{remark}[Small-noise decomposition and large deviations]
\label{rem:OM-LDP-relation}
To make the dependence on the noise intensity explicit, we replace
$\sigma_t$ by $\sqrt{\epsilon}\,\sigma_t$ and consider
\begin{equation}
\label{sde}
    dX_t^\epsilon
    =
    b_t(X_t^\epsilon)\,dt
    +
    \sqrt{\epsilon}\,\sigma_t\,dB_t^H,
    \qquad
    X_0^\epsilon=x_0,
\end{equation}
where $\epsilon>0$ and the coefficients satisfy
Assumption~\ref{ass:A}.

The associated Onsager--Machlup functional can be decomposed into a
fractional kinetic term and a potential-type correction. Define
\begin{equation}
\label{eq:fractional-kinetic-term}
    I_H(\phi)
    :=
    \frac12
    \int_0^T
    \left|
        (K_H^\sigma)^{-1}
        \left(
            \phi_\cdot-x_0
            -
            \int_0^\cdot b_s(\phi_s)\,ds
        \right)(t)
    \right|^2\,dt
\end{equation}
and
\begin{equation}
\label{eq:potential-term}
    V(\phi)
    :=
    -\frac{d_H}{2}
    \int_0^T
    \nabla\!\cdot b_t(\phi_t)\,dt.
\end{equation}
Then
\begin{equation}
\label{eq:OM-kinetic-potential-decomposition}
    J_\epsilon(\phi)
    =
    \frac{1}{\epsilon}I_H(\phi)-V(\phi).
\end{equation}

The kinetic term $I_H$ coincides with the rate functional in the
corresponding Freidlin--Wentzell large deviation principle. Informally,
on the logarithmic scale,
\begin{equation*}
    \mathbb P\bigl(X^\epsilon\approx\phi\bigr)
    \approx
    \exp\left\{
        -\frac{I_H(\phi)}{\epsilon}
    \right\}.
\end{equation*}
Thus, $I_H(\phi)$ measures the leading exponential cost of observing a
path close to $\phi$ in the small-noise limit. In particular,
\begin{equation*}
    I_H(\phi)\geq0,
    \qquad
    I_H(\phi)=0
    \quad\Longleftrightarrow\quad
    \phi=\bar X,
\end{equation*}
where $\bar X$ is the solution of the noise-free system
\eqref{nnoise} with initial condition $x_0$.

Consequently, the term $\epsilon^{-1}I_H$ dominates the potential
correction $V$ as $\epsilon\to0$, suggesting that global most probable
paths should approach $\bar X$. The persistence properties and
quantitative convergence rates are established below by direct
variational estimates.
\end{remark}

\section{Most Probable Path Preservation for Small Noise}
\label{sec:most probable transition path}

In this section, we study the persistence of most probable paths of
\eqref{sde} in the small-noise regime and establish quantitative
convergence rates of global most probable paths toward the
corresponding trajectories of the noise-free system.

    According to \eqref{eq:OM-functional-n-dimensional}, the
Onsager--Machlup functional associated with \eqref{sde} is
\begin{equation}
\label{om-functional}
\begin{aligned}
    J_\epsilon(\phi)
    &=
    \frac{1}{2\epsilon}
    \int_0^T
    \left|
        (K_H^\sigma)^{-1}
        \left(
            \phi_\cdot-x_0
            -
            \int_0^\cdot b_s(\phi_s)\,ds
        \right)(t)
    \right|^2\,dt
    \\
    &\quad+
    \frac{d_H}{2}
    \int_0^T
    \nabla\!\cdot b_t(\phi_t)\,dt,
    \qquad
    \phi\in\mathcal H_{H;x_0,x_T}^{\sigma,T}.
\end{aligned}
\end{equation}

For prescribed endpoints $x_0,x_T\in\mathbb R^n$, define the
endpoint-constrained admissible class by
\begin{equation*}
    \mathcal H_{H;x_0,x_T}^{\sigma,T}
    :=
    \left\{
        \phi:
        \phi_0=x_0,\;
        \phi_T=x_T,\;
        \phi_\cdot-x_0
        \in
        \mathcal H_H^\sigma([0,T];\mathbb R^n)
    \right\}.
\end{equation*}
The most probable transition paths from $x_0$ to $x_T$ are the
minimizers of the Onsager--Machlup functional
\eqref{om-functional} over
$\mathcal H_{H;x_0,x_T}^{\sigma,T}$.

We next introduce two classes of most probable paths.

\begin{definition}[Most probable transition path]
\label{def:most-probable-transition-path}
A path
\begin{equation*}
    \phi^*
    \in
    \mathcal H_{H;x_0,x_T}^{\sigma,T}
\end{equation*}
is called a most probable transition path of \eqref{sde} from $x_0$ to
$x_T$ over $[0,T]$ if
\begin{equation*}
    J_\epsilon(\phi^*)
    =
    \inf_{\phi\in\mathcal H_{H;x_0,x_T}^{\sigma,T}}
    J_\epsilon(\phi).
\end{equation*}
\end{definition}

\begin{definition}[Most probable evolution path]
\label{def:most-probable-evolution-path}
A path
\begin{equation*}
    \phi^*
    \in
    x_0+\mathcal H_H^\sigma
\end{equation*}
is called a most probable evolution path of \eqref{sde} starting from
$x_0$ over $[0,T]$ if
\begin{equation*}
    J_\epsilon(\phi^*)
    =
    \inf_{
        \phi\in
        x_0+\mathcal H_H^\sigma
    }
    J_\epsilon(\phi).
\end{equation*}
\end{definition}

We now consider the corresponding noise-free system
\begin{equation}
\label{nnoise}
    \bar X_t'
    =
    b_t(\bar X_t),
    \qquad
    \bar X_0=x_0,
\end{equation}
and, for the fixed-endpoint problem, assume that
\begin{equation*}
    \bar X_T=x_T.
\end{equation*}

\begin{definition}[Persistence of a most probable path]
\label{def:persistence-most-probable-path}
The deterministic trajectory $\bar X$ is said to persist as a most
probable transition path or evolution path of \eqref{sde} at noise
intensity $\epsilon$ if
\begin{equation*}
    J_\epsilon(\bar X)
    =
    \inf_{\phi\in\mathcal A}J_\epsilon(\phi),
\end{equation*}
where
\begin{equation*}
    \mathcal A
    =
    \mathcal H_{H;x_0,x_T}^{\sigma,T}
    \quad\text{or}\quad
    x_0+\mathcal H_H^\sigma,
\end{equation*}
respectively.
\end{definition}

Since a most probable transition path is subject to the additional
terminal constraint,
\begin{equation*}
    \mathcal H_{H;x_0,x_T}^{\sigma,T}
    \subset
    x_0+\mathcal H_H^\sigma.
\end{equation*}
Consequently, persistence as a most probable evolution path implies
persistence as a most probable transition path, whereas the converse
need not hold.

\subsection{First and second variations}

To characterize most probable paths, we study the variations of
\eqref{om-functional} over the two admissible path spaces. We first
consider the fixed-endpoint problem, whose admissible perturbation
space is
\begin{equation*}
    \mathcal H_{H,0}^\sigma
    :=
    \left\{
        \eta\in
        \mathcal H_H^\sigma\bigl([0,T];\mathbb R^n\bigr):
        \eta_0=\eta_T=0
    \right\}.
\end{equation*}

Let $\eta\in\mathcal H_{H,0}^\sigma$. Since $\bar X$ satisfies
\eqref{nnoise}, Taylor expansion of the drift gives
\begin{align*}
&\bar X_\cdot+\lambda\eta_\cdot-x_0
-\int_0^\cdot b_s(\bar X_s+\lambda\eta_s)\,ds
\\
&\quad=
\lambda
\left(
    \eta_\cdot
    -
    \int_0^\cdot
    \nabla b_s(\bar X_s)\eta_s\,ds
\right)
-
\frac{\lambda^2}{2}
\int_0^\cdot
\eta_s^T\nabla^2b_s(\bar X_s)\eta_s\,ds
+
o(\lambda^2).
\end{align*}
Consequently,
\begin{equation*}
\begin{aligned}
I_H(\bar X+\lambda\eta)
=
\frac{\lambda^2}{2}
\int_0^T
\left|
(K_H^\sigma)^{-1}
\left(
    \eta_\cdot
    -
    \int_0^\cdot
    \nabla b_s(\bar X_s)\eta_s\,ds
\right)(t)
\right|^2\,dt
+
o(\lambda^2).
\end{aligned}
\end{equation*}
Similarly,
\begin{equation*}
\begin{aligned}
V(\bar X+\lambda\eta)
={}&
V(\bar X)
-
\frac{d_H\lambda}{2}
\int_0^T
\nabla(\nabla\!\cdot b_t)(\bar X_t)^T\eta_t\,dt
\\
&-
\frac{d_H\lambda^2}{4}
\int_0^T
\eta_t^T
\nabla^2(\nabla\!\cdot b_t)(\bar X_t)
\eta_t\,dt
+
o(\lambda^2).
\end{aligned}
\end{equation*}
It follows that
\begin{equation}
\label{eq:J-taylor-variation}
    J_\epsilon(\bar X+\eta)
    =
    J_\epsilon(\bar X)
    +
    \delta J_\epsilon(\bar X)[\eta]
    +
    \frac12
    \delta^2J_\epsilon(\bar X)[\eta,\eta]
    +
    o\left(
        \|\eta\|_{\mathcal H_H^\sigma}^2
    \right),
\end{equation}
where
\begin{equation}
\label{eq:first-variation-barX}
    \delta J_\epsilon(\bar X)[\eta]
    =
    \frac{d_H}{2}
    \int_0^T
    \nabla(\nabla\!\cdot b_t)(\bar X_t)^T\eta_t\,dt
\end{equation}
and
\begin{equation}
\label{eq:second-variation-barX}
\begin{aligned}
\delta^2J_\epsilon(\bar X)[\eta,\eta]
={}&
\frac{1}{\epsilon}
\int_0^T
\left|
(K_H^\sigma)^{-1}
\left(
    \eta_\cdot
    -
    \int_0^\cdot
    \nabla b_s(\bar X_s)\eta_s\,ds
\right)(t)
\right|^2\,dt
\\
&+
\frac{d_H}{2}
\int_0^T
\eta_t^T
\nabla^2(\nabla\!\cdot b_t)(\bar X_t)
\eta_t\,dt.
\end{aligned}
\end{equation}

For the most probable evolution-path problem, the admissible
perturbations belong to
$\mathcal H_H^\sigma([0,T];\mathbb R^n)$ and are not required to
vanish at $T$. Since the kinetic residual vanishes at $\bar X$, the
same calculation applies without additional boundary terms.
Therefore, \eqref{eq:first-variation-barX} and
\eqref{eq:second-variation-barX} remain valid, but on the larger
free-endpoint perturbation space.

A necessary condition for $\bar X$ to be a local minimizer in either
problem is
\begin{equation}
\label{cond1}
    \nabla(\nabla\!\cdot b_t)(\bar X_t)=0,
    \qquad
    0\leq t\leq T.
\end{equation}
Under \eqref{cond1},
\begin{equation}
\label{first-variation}
    \delta J_\epsilon(\bar X)[\eta]=0
\end{equation}
for every admissible perturbation in the corresponding path space.
Local minimality may then be established by proving the coercivity of
\eqref{eq:second-variation-barX} on that perturbation space.

\subsection{Persistence and loss of local minimality of the deterministic path}

We first give a coercivity criterion for the local minimality of the
deterministic path.

\begin{theorem}
\label{variation}
Let $1/4<H<1$, and let $\bar X$ solve \eqref{nnoise}. Suppose that
\eqref{cond1} holds and that there exists $c_0>0$ such that
\begin{equation}
\label{cond2}
    \delta^2J_\epsilon(\bar X)[\eta,\eta]
    \geq
    c_0
    \|\eta\|_{\mathcal H_H^\sigma([0,T];\mathbb R^n)}^2.
\end{equation}
If \eqref{cond2} holds for every
$\eta\in\mathcal H_{H,0}^\sigma([0,T];\mathbb R^n)$, then $\bar X$ is
a strict local minimizer under the fixed-endpoint constraint. If it
holds for every
$\eta\in\mathcal H_H^\sigma([0,T];\mathbb R^n)$, then $\bar X$ is a
strict local minimizer in the free-endpoint path space.
\end{theorem}

\begin{proof}
By \eqref{cond1}, the first variation vanishes. Hence,
\eqref{eq:J-taylor-variation} and \eqref{cond2} give
\begin{equation*}
\begin{aligned}
    J_\epsilon(\bar X+\eta)-J_\epsilon(\bar X)
    &=
    \frac12\delta^2J_\epsilon(\bar X)[\eta,\eta]
    +
    o\left(
        \|\eta\|_{\mathcal H_H^\sigma}^2
    \right)
    \\
    &\geq
    \frac{c_0}{2}
    \|\eta\|_{\mathcal H_H^\sigma}^2
    +
    o\left(
        \|\eta\|_{\mathcal H_H^\sigma}^2
    \right).
\end{aligned}
\end{equation*}
The right-hand side is positive for all sufficiently small nonzero
admissible perturbations, proving the claim.
\end{proof}

\begin{remark}
Condition~\eqref{cond1} means that the local volume expansion rate
$\nabla\!\cdot b_t$ is spatially stationary to first order along
$\bar X$. Condition~\eqref{cond2} requires the fractional kinetic
energy to dominate the second-order contribution of the divergence,
thereby ensuring the local stability of $\bar X$.
\end{remark}

To verify \eqref{cond2}, define the linearized fractional operator
\begin{equation}
\label{eq:AH-definition}
    \mathcal A_H\eta(s)
    :=
    (K_H^\sigma)^{-1}
    \left(
        \eta_\cdot
        -
        \int_0^\cdot
        \nabla b_u(\bar X_u)\eta_u\,du
    \right)(s)
\end{equation}
and the quadratic forms
\begin{equation}
\label{eq:kinetic-part}
    \mathcal I_H(\eta)
    :=
    \|\mathcal A_H\eta\|_{L^2(0,T)}^2
\end{equation}
and
\begin{equation}
\label{eq:potential-part}
    \mathcal V(\eta)
    :=
    -
    \int_0^T
    \eta_s^T
    \nabla^2(\nabla\!\cdot b_s)(\bar X_s)
    \eta_s\,ds.
\end{equation}
Then
\begin{equation}
\label{eq:second-variation-decomposition}
    \delta^2J_\epsilon(\bar X)[\eta,\eta]
    =
    \frac{1}{\epsilon}\mathcal I_H(\eta)
    -
    \frac{d_H}{2}\mathcal V(\eta).
\end{equation}

The following lemma shows that $\mathcal I_H$ is equivalent to the
Cameron--Martin norm.

\begin{lemma}
\label{lem:IH-HH-equivalence}
There exist constants $C_1,C_2>0$ such that
\begin{equation}
\label{eq:IH-HH-equivalence}
    C_1
    \|\eta\|_{\mathcal H_H^\sigma(0,T)}^2
    \leq
    \mathcal I_H(\eta)
    \leq
    C_2
    \|\eta\|_{\mathcal H_H^\sigma(0,T)}^2
\end{equation}
for every
$\eta\in\mathcal H_H^\sigma([0,T];\mathbb R^n)$. In particular, the
same estimate holds on
$\mathcal H_{H,0}^\sigma([0,T];\mathbb R^n)$.
\end{lemma}

\begin{proof}
Define
\begin{equation*}
    (\mathcal K\eta)(t)
    :=
    \int_0^t
    \nabla b_u(\bar X_u)\eta_u\,du
\end{equation*}
and set
\begin{equation*}
    f:=(K_H^\sigma)^{-1}\eta,
    \qquad
    \widetilde{\mathcal K}
    :=
    (K_H^\sigma)^{-1}\mathcal K K_H^\sigma.
\end{equation*}
Then
\begin{equation*}
    \mathcal I_H(\eta)
    =
    \left\|
        (I-\widetilde{\mathcal K})f
    \right\|_{L^2(0,T)}^2,
    \qquad
    \|\eta\|_{\mathcal H_H^\sigma(0,T)}
    =
    \|f\|_{L^2(0,T)}.
\end{equation*}
The Volterra structure of $K_H^\sigma$ and $\mathcal K$ implies that
$\widetilde{\mathcal K}$ is bounded and
$\sigma(\widetilde{\mathcal K})=\{0\}$. Thus
$I-\widetilde{\mathcal K}$ is boundedly invertible on
$L^2([0,T];\mathbb R^n)$, and
\begin{equation*}
    C_1\|f\|_{L^2(0,T)}^2
    \leq
    \left\|
        (I-\widetilde{\mathcal K})f
    \right\|_{L^2(0,T)}^2
    \leq
    C_2\|f\|_{L^2(0,T)}^2.
\end{equation*}
This proves \eqref{eq:IH-HH-equivalence}.
\end{proof} 

The continuous embedding of the Cameron--Martin space into $L^2$
gives the following estimate.

\begin{lemma}[Cameron--Martin embedding estimate]
\label{lem:HH-poincare}
There exists $\lambda_0>0$ such that
\begin{equation}
\label{eq:HH-poincare}
    \|\eta\|_{\mathcal H_H^\sigma(0,T)}^2
    \geq
    \lambda_0\|\eta\|_{L^2(0,T)}^2
\end{equation}
for every
$\eta\in\mathcal H_H^\sigma([0,T];\mathbb R^n)$.
\end{lemma}

\begin{proof}
Writing $\eta=K_H^\sigma f$ and using the boundedness of
$K_H^\sigma$ on $L^2$, we obtain
\begin{equation*}
    \|\eta\|_{L^2(0,T)}
    \leq
    C\|f\|_{L^2(0,T)}
    =
    C\|\eta\|_{\mathcal H_H^\sigma(0,T)}.
\end{equation*}
The result follows with $\lambda_0=C^{-2}$.
\end{proof}

We next control the potential term by the Cameron--Martin norm.

\begin{lemma}
\label{lem:potential-upper-bound}
There exists $\lambda_V\geq0$ such that
\begin{equation}
\label{eq:potential-upper-bound}
    \mathcal V(\eta)
    \leq
    \lambda_V
    \|\eta\|_{\mathcal H_H^\sigma(0,T)}^2
\end{equation}
for every
$\eta\in\mathcal H_H^\sigma([0,T];\mathbb R^n)$.
\end{lemma}

\begin{proof}
Set
\begin{equation}
\label{eq:lambda-V-definition}
    \mu_V
    :=
    \max\left\{
        0,\,
        \sup_{s\in[0,T]}
        \lambda_{\max}
        \left(
            -\nabla^2(\nabla\!\cdot b_s)(\bar X_s)
        \right)
    \right\}.
\end{equation}
The regularity of $b$ implies $\mu_V<\infty$, and hence
\begin{equation*}
    \mathcal V(\eta)
    \leq
    \mu_V\|\eta\|_{L^2(0,T)}^2
    \leq
    \frac{\mu_V}{\lambda_0}
    \|\eta\|_{\mathcal H_H^\sigma(0,T)}^2
\end{equation*}
by Lemma~\ref{lem:HH-poincare}. Thus the result holds with
$\lambda_V=\mu_V/\lambda_0$.
\end{proof}

Combining the preceding estimates gives the small-noise persistence
criterion.

\begin{theorem}[Persistence of local minimizers]
\label{thm:local-minimizer-small-noise}
Let $1/4<H<1$, and let $\bar X$ solve \eqref{nnoise}. Suppose that
\eqref{cond1} holds and
\begin{equation}
\label{eq:epsilon-condition-local}
    \frac{C_1}{\epsilon}
    -
    \frac{d_H}{2}\lambda_V
    >
    0.
\end{equation}
Then $\bar X$ is a strict local minimizer of $J_\epsilon$ in both the
fixed-endpoint and free-endpoint path spaces. In particular, it
persists as both a locally most probable transition path and a locally
most probable evolution path.
\end{theorem}

\begin{proof}
By Lemmas~\ref{lem:IH-HH-equivalence} and
\ref{lem:potential-upper-bound},
\begin{align}
    \delta^2J_\epsilon(\bar X)[\eta,\eta]
    &=
    \frac{1}{\epsilon}\mathcal I_H(\eta)
    -
    \frac{d_H}{2}\mathcal V(\eta)
    \notag\\
    &\geq
    \left(
        \frac{C_1}{\epsilon}
        -
        \frac{d_H}{2}\lambda_V
    \right)
    \|\eta\|_{\mathcal H_H^\sigma(0,T)}^2.
    \label{eq:second-variation-coercive}
\end{align}
Thus \eqref{eq:epsilon-condition-local} makes the second variation
coercive on $\mathcal H_H^\sigma([0,T];\mathbb R^n)$ and hence also on
$\mathcal H_{H,0}^\sigma([0,T];\mathbb R^n)$. The conclusion follows
from Theorem~\ref{variation}.
\end{proof}

\begin{remark}
\label{rem:constants-H-sigma}
The constants satisfy
\begin{equation*}
    \lambda_0=\lambda_0(H,\sigma,T),
    \qquad
    C_1=C_1(H,\sigma,T,\nabla b,\bar X),
\end{equation*}
while $\lambda_V$ also depends on
$\nabla^2(\nabla\!\cdot b)$. This dependence reflects the nonlocal
Cameron--Martin geometry induced by the fractional noise.
\end{remark}

In contrast, a negative direction of the Hessian of the divergence
destroys local minimality when the noise intensity is sufficiently
large.

\begin{theorem}[Destruction of local minimality for large noise]
\label{thm:destruction-large-noise}
Let $1/4<H<1$, and let $\bar X$ solve \eqref{nnoise}. Suppose that
\eqref{cond1} holds and that, for some $t_0\in(0,T)$,
\begin{equation}
\label{eq:negative-Hessian-direction}
    \nabla^2(\nabla\!\cdot b_{t_0})(\bar X_{t_0})
\end{equation}
has a negative eigenvalue. Then there exists $\epsilon_*>0$ such that,
for every $\epsilon>\epsilon_*$, the second variation of $J_\epsilon$
at $\bar X$ has a negative direction. Consequently, $\bar X$ is
neither a local most probable transition path nor a local most
probable evolution path.
\end{theorem}

\begin{proof}
Let $v_0\neq0$ be a negative direction of
\eqref{eq:negative-Hessian-direction}. By continuity, there exists an
interval $U\subset(0,T)$ containing $t_0$ such that
\begin{equation*}
    v_0^T
    \nabla^2(\nabla\!\cdot b_t)(\bar X_t)
    v_0<0,
    \qquad t\in U.
\end{equation*}
Choose a nonzero $\chi\in C_c^\infty(U)$ and set
$\eta_t=\chi(t)v_0$. Then
$\eta\in\mathcal H_{H,0}^\sigma([0,T];\mathbb R^n)$ and
\begin{equation*}
    \int_0^T
    \eta_t^T
    \nabla^2(\nabla\!\cdot b_t)(\bar X_t)
    \eta_t\,dt
    <0.
\end{equation*}
It follows from \eqref{eq:second-variation-barX} that
\begin{equation*}
    \delta^2J_\epsilon(\bar X)[\eta,\eta]<0
\end{equation*}
for all sufficiently large $\epsilon$. Since \eqref{cond1} implies
$\delta J_\epsilon(\bar X)=0$, Taylor's expansion shows that $\bar X$
is not a local minimizer. The same perturbation is admissible in both
the fixed-endpoint and free-endpoint path spaces.
\end{proof}

\subsection{Global minimality for sufficiently small noise}

We now strengthen the local persistence result to unique global
minimality.

\begin{theorem}[Unique global minimality for sufficiently small noise]
\label{thm:global-minimality-small-noise}
Let $1/4<H<1$, and let $\bar X$ solve \eqref{nnoise}. Suppose that
\eqref{cond1} holds. Then there exists $\epsilon_*>0$ such that, for
every $0<\epsilon<\epsilon_*$, the path $\bar X$ is the unique global
minimizer of $J_\epsilon$ over
$x_0+\mathcal H_H^\sigma([0,T];\mathbb R^n)$. Consequently, it is also
the unique global minimizer over
$\mathcal H_{H;x_0,x_T}^{\sigma,T}$, where $x_T=\bar X_T$.
\end{theorem}

\begin{proof}
Choose $\epsilon_{\mathrm{loc}}>0$ satisfying
\begin{equation*}
    \frac{C_1}{\epsilon_{\mathrm{loc}}}
    -
    \frac{d_H}{2}\lambda_V
    >
    0.
\end{equation*}
By Theorem~\ref{thm:local-minimizer-small-noise}, there exists
$\delta>0$ such that
\begin{equation}
\label{eq:local-minimality}
    J_{\epsilon_{\mathrm{loc}}}(\bar X)
    <
    J_{\epsilon_{\mathrm{loc}}}(\phi)
\end{equation}
whenever
\begin{equation*}
    0<
    \|\phi-\bar X\|_{\mathcal H_H^\sigma(0,T)}
    \leq\delta.
\end{equation*}
Since $I_H(\bar X)=0$, for every
$0<\epsilon\leq\epsilon_{\mathrm{loc}}$,
\begin{align*}
    J_\epsilon(\phi)-J_\epsilon(\bar X)
    ={}&
    J_{\epsilon_{\mathrm{loc}}}(\phi)
    -
    J_{\epsilon_{\mathrm{loc}}}(\bar X)
    \\
    &+
    \left(
        \frac1\epsilon
        -
        \frac1{\epsilon_{\mathrm{loc}}}
    \right)
    I_H(\phi).
\end{align*}
Thus \eqref{eq:local-minimality} remains valid, with the same
$\delta$, for every $0<\epsilon\leq\epsilon_{\mathrm{loc}}$.

Define
\begin{equation}
\label{eq:Ndelta-def}
    \mathcal N_\delta
    :=
    \left\{
        \phi\in x_0+\mathcal H_H^\sigma:
        \|\phi-\bar X\|_{\mathcal H_H^\sigma(0,T)}
        \leq\delta
    \right\}
\end{equation}
and, for $M>0$,
\begin{equation}
\label{eq:UM-def}
    U(M)
    :=
    \left\{
        \phi\in x_0+\mathcal H_H^\sigma:
        I_H(\phi)\leq M
    \right\}.
\end{equation}

We claim that a sufficiently small kinetic-energy sublevel set is
contained in $\mathcal N_\delta$. Indeed, set
$h:=\phi-\bar X$. Since $\bar X$ solves \eqref{nnoise},
\begin{align}
& (K_H^\sigma)^{-1}
\left(
    \phi_\cdot-x_0
    -
    \int_0^\cdot b_s(\phi_s)\,ds
\right)
\notag\\
&\quad=
(K_H^\sigma)^{-1}
\left(
    h_\cdot
    -
    \int_0^\cdot
    \bigl[
        b_s(\bar X_s+h_s)-b_s(\bar X_s)
    \bigr]\,ds
\right).
\label{eq:residual-difference}
\end{align}
The global Lipschitz continuity of $b$ and the uniform Volterra
resolvent estimate imply that
\begin{equation}
\label{eq:H-distance-controlled-by-IH}
    \|\phi-\bar X\|_{\mathcal H_H^\sigma(0,T)}
    \leq
    C_R\sqrt{2I_H(\phi)},
\end{equation}
where $C_R>0$ is independent of $\phi$. Choosing
\begin{equation}
\label{eq:M0-def}
    M_0
    :=
    \frac{\delta^2}{2C_R^2},
\end{equation}
we obtain
\begin{equation*}
    U(M_0)\subset\mathcal N_\delta.
\end{equation*}
Therefore,
\begin{equation}
\label{eq:min-on-UM0}
    J_\epsilon(\bar X)
    <
    J_\epsilon(\phi),
    \qquad
    \phi\in U(M_0),\quad
    \phi\neq\bar X,
\end{equation}
for every $0<\epsilon\leq\epsilon_{\mathrm{loc}}$.

It remains to consider paths outside $U(M_0)$. Since
$\nabla\!\cdot b$ is bounded, set
\begin{equation*}
    C_{\mathrm{div}}
    :=
    \frac{d_H}{2}
    T
    \|\nabla\!\cdot b\|_\infty.
\end{equation*}
Then $|V(\phi)|\leq C_{\mathrm{div}}$ for every admissible $\phi$.
If $\phi\notin U(M_0)$, then
\begin{equation}
\label{eq:J-lower-outside}
    J_\epsilon(\phi)
    >
    \frac{M_0}{\epsilon}
    -
    C_{\mathrm{div}},
\end{equation}
whereas
\begin{equation*}
    J_\epsilon(\bar X)
    =
    -V(\bar X)
    \leq
    C_{\mathrm{div}}.
\end{equation*}
Hence, with
\begin{equation}
\label{eq:eps-global}
    \epsilon_{\mathrm{glob}}
    :=
    \frac{M_0}{2C_{\mathrm{div}}+1},
\end{equation}
we have
\begin{equation*}
    J_\epsilon(\phi)>J_\epsilon(\bar X),
    \qquad
    \phi\notin U(M_0),
\end{equation*}
whenever $0<\epsilon<\epsilon_{\mathrm{glob}}$.

Finally, set
\begin{equation*}
    \epsilon_*
    :=
    \min
    \left\{
        \epsilon_{\mathrm{loc}},
        \epsilon_{\mathrm{glob}}
    \right\}.
\end{equation*}
For every $0<\epsilon<\epsilon_*$, the estimates inside and outside
$U(M_0)$ show that
\begin{equation*}
    J_\epsilon(\bar X)<J_\epsilon(\phi)
    \qquad
    \text{for every }\phi\neq\bar X.
\end{equation*}
Thus $\bar X$ is the unique global minimizer in the free-endpoint path
space. Since
\begin{equation*}
    \mathcal H_{H;x_0,x_T}^{\sigma,T}
    \subset
    x_0+\mathcal H_H^\sigma,
\end{equation*}
the fixed-endpoint conclusion follows immediately.
\end{proof}

For vector fields with spatially constant divergence, persistence holds
for every noise intensity.

\begin{corollary}[Spatially constant divergence]
\label{cor: spatially constant divergence}
Suppose that the drift of \eqref{sde} satisfies
\begin{equation*}
    \nabla\!\cdot b_t(x)=C_t,
    \qquad
    (t,x)\in[0,T]\times\mathbb R^n.
\end{equation*}
Then, for every initial condition and every $\epsilon>0$, the
corresponding deterministic trajectory persists as both the unique
most probable evolution path and the unique most probable transition
path.
\end{corollary}

\begin{proof}
Since the divergence term is independent of the path,
\begin{equation*}
    J_\epsilon(\phi)
    =
    \frac1\epsilon I_H(\phi)
    +
    \frac{d_H}{2}\int_0^T C_t\,dt.
\end{equation*}
Hence $J_\epsilon$ is minimized precisely when $I_H(\phi)=0$, or
equivalently,
\begin{equation*}
    \phi_t
    =
    x_0+\int_0^t b_s(\phi_s)\,ds.
\end{equation*}
Uniqueness of the deterministic system then gives $\phi=\bar X$.
\end{proof}

\subsection{Quantitative small-noise convergence of most probable  paths}
\label{subsec:quantitative-convergence}

In the preceding subsections, we showed that, under condition
\eqref{cond1} and for sufficiently small noise intensity, the
deterministic trajectory $\bar X$ persists as a most probable path of
\eqref{sde}. We now consider the general case in which \eqref{cond1}
may fail.

By \eqref{eq:first-variation-barX}, condition \eqref{cond1} is
necessary and sufficient for $\bar X$ to be a critical point of
$J_\epsilon$. Hence, if \eqref{cond1} fails, $\bar X$ cannot be a most
probable path for any fixed $\epsilon>0$. Nevertheless, the
decomposition
\begin{equation*}
    J_\epsilon(\phi)
    =
    \frac1\epsilon I_H(\phi)-V(\phi)
\end{equation*}
suggests that global most probable paths should approach $\bar X$ as
$\epsilon\to0$, because $I_H$ is nonnegative and vanishes uniquely at
$\bar X$. This is also consistent with the large deviation principle
in \cite{FanYuYuan2023}, whose rate functional is $I_H$. The large
deviation principle, however, does not by itself imply convergence of
the minimizers of $J_\epsilon$, nor does it provide a convergence rate
in the path-space topologies considered here.

We prove below that every global most probable path converges to
$\bar X$ in both the uniform and H\"older topologies at the rate
$O(\epsilon)$.

Since
$b\in C_b^{1,3}([0,T]\times\mathbb R^n;\mathbb R^n)$, there exists
$L>0$ such that
\begin{equation}
\label{eq:global-Lipschitz-b}
    |b_t(x)-b_t(y)|
    \leq
    L|x-y|,
    \qquad
    t\in[0,T],\quad x,y\in\mathbb R^n.
\end{equation}
The boundedness of $\nabla(\nabla\!\cdot b_t)$ also gives
\begin{equation}
\label{eq:potential-Lipschitz}
    |V(\phi)-V(\psi)|
    \leq
    L\int_0^T|\phi_t-\psi_t|\,dt
\end{equation}
for all admissible paths $\phi$ and $\psi$, after increasing $L$ if
necessary.

\begin{theorem}
\label{thm:uniform-convergence-global-minimizers}
Let $1/4<H<1$, and let $\bar X$ solve \eqref{nnoise}. Let
$\phi^\epsilon$ be a global minimizer of $J_\epsilon$ in either the
fixed-endpoint or free-endpoint admissible path space. Then there
exists $C>0$, independent of $\epsilon$, such that
\begin{equation}
\label{eq:uniform-convergence-rate}
    \|\phi^\epsilon-\bar X\|_\infty
    \leq
    C\epsilon.
\end{equation}
\end{theorem}

\begin{proof}
Set
\begin{equation*}
    h^\epsilon:=\phi^\epsilon-\bar X
\end{equation*}
and
\begin{equation*}
    u^\epsilon
    :=
    (K_H^\sigma)^{-1}
    \left(
        \phi^\epsilon_\cdot-x_0
        -
        \int_0^\cdot b_s(\phi^\epsilon_s)\,ds
    \right).
\end{equation*}
Then
\begin{equation*}
    I_H(\phi^\epsilon)
    =
    \frac12\|u^\epsilon\|_{L^2}^2
\end{equation*}
and
\begin{equation}
\label{eq:path-difference-control}
    h^\epsilon_t
    =
    (K_H^\sigma u^\epsilon)(t)
    +
    \int_0^t
    \left[
        b_s(\phi^\epsilon_s)-b_s(\bar X_s)
    \right]ds.
\end{equation}
By \eqref{eq:global-Lipschitz-b}, Gr\"onwall's inequality, and the
continuous Cameron--Martin embedding,
\begin{equation}
\label{eq:uniform-control-by-u}
    \|h^\epsilon\|_\infty
    \leq
    C\|u^\epsilon\|_{L^2}.
\end{equation}

Since $\bar X$ is admissible, $I_H(\bar X)=0$, and
$\phi^\epsilon$ is a global minimizer,
\begin{align*}
    \frac12\|u^\epsilon\|_{L^2}^2
    &=
    I_H(\phi^\epsilon)
    \\
    &\leq
    \epsilon
    |V(\phi^\epsilon)-V(\bar X)|
    \\
    &\leq
    \epsilon L
    \int_0^T|h^\epsilon_t|\,dt
    \\
    &\leq
    C\epsilon\|u^\epsilon\|_{L^2}.
\end{align*}
Therefore,
\begin{equation}
\label{eq:u-epsilon-order}
    \|u^\epsilon\|_{L^2}
    \leq
    C\epsilon.
\end{equation}
Combining this with \eqref{eq:uniform-control-by-u} proves
\eqref{eq:uniform-convergence-rate}.
\end{proof}

We next obtain convergence in the H\"older topology.

\begin{corollary}
\label{cor:Holder-convergence-global-minimizers}
Under the assumptions of
Theorem~\ref{thm:uniform-convergence-global-minimizers}, let
$\beta\in(0,H)$ be such that
\begin{equation*}
    \mathcal H_H^\sigma
    \hookrightarrow
    C^\beta([0,T];\mathbb R^n)
\end{equation*}
continuously. Then there exists $C>0$, independent of $\epsilon$, such
that
\begin{equation}
\label{eq:Holder-convergence-rate}
    \|\phi^\epsilon-\bar X\|_\beta
    \leq
    C\epsilon.
\end{equation}
\end{corollary}

\begin{proof}
Write
\begin{equation}
\label{eq:Holder-decomposition}
    h^\epsilon_t
    =
    (K_H^\sigma u^\epsilon)(t)
    +
    \int_0^t g^\epsilon_s\,ds,
    \qquad
    g^\epsilon_t
    :=
    b_t(\bar X_t+h^\epsilon_t)-b_t(\bar X_t).
\end{equation}
By Theorem~\ref{thm:uniform-convergence-global-minimizers},
\begin{equation*}
    \|g^\epsilon\|_\infty
    \leq
    L\|h^\epsilon\|_\infty
    \leq
    C\epsilon.
\end{equation*}
Hence
\begin{equation*}
    \left\|
        \int_0^\cdot g^\epsilon_s\,ds
    \right\|_\beta
    \leq
    C\|g^\epsilon\|_\infty
    \leq
    C\epsilon.
\end{equation*}
Moreover, the Cameron--Martin embedding and
\eqref{eq:u-epsilon-order} give
\begin{equation*}
    \|K_H^\sigma u^\epsilon\|_\beta
    \leq
    C\|u^\epsilon\|_{L^2}
    \leq
    C\epsilon.
\end{equation*}
The conclusion follows from \eqref{eq:Holder-decomposition}.
\end{proof}

Thus, even when \eqref{cond1} fails, every global most probable path
converges to $\bar X$ at the rate $O(\epsilon)$ in both the uniform
and H\"older topologies.

\section{Long-Time Behavior of the Second Variation along Periodic
Trajectories}
\label{sec:long-time-periodic-trajectories}

In the preceding section, we showed that, on every fixed time interval
$[0,T]$, most probable paths persist below a noise threshold that may
depend on $T$. We now investigate whether this persistence remains
valid over long time intervals. More precisely, for a fixed noise
intensity, we study the variational structure along a periodic orbit
over $[0,NT]$ as $N\to\infty$, from the viewpoint of most probable
paths.

Throughout this section, we assume that the drift is autonomous and
that the noise-free system admits a nonconstant hyperbolic
$T$-periodic orbit $\bar X$:
\begin{equation*}
    \bar X_{t+T}=\bar X_t,
    \qquad
    \bar X_0=\bar X_T=x_0.
\end{equation*}
We also assume that $\sigma_t$ is $T$-periodic and regard $\bar X$ as
a trajectory on $[0,NT]$ for every $N\in\mathbb N^+$. The
corresponding Onsager--Machlup functional is
\begin{equation}
\label{eq:om functional longtime}
\begin{aligned}
    J_{\epsilon,N}(\phi)
    &=
    \frac{1}{2\epsilon}
    \int_0^{NT}
    \left|
        (K_H^\sigma)^{-1}
        \left(
            \phi_\cdot-x_0
            -
            \int_0^\cdot b(\phi_s)\,ds
        \right)(t)
    \right|^2\,dt
    \\
    &\quad+
    \frac{d_H}{2}
    \int_0^{NT}
    \nabla\!\cdot b(\phi_t)\,dt,
\end{aligned}
\end{equation}
where $\phi_0=\phi_{NT}=x_0$.

For every fixed $N$, the results of the preceding section provide a
threshold $\epsilon_N^*>0$ such that $\bar X$ is a local minimizer
whenever $0<\epsilon<\epsilon_N^*$. The long-time problem is whether
this threshold can be chosen independently of $N$.

The second variation along $\bar X$ has the form
\begin{equation*}
    \delta^2J_{\epsilon,N}(\bar X)[\eta,\eta]
    =
    \frac{1}{\epsilon}\mathcal I_{H,N}(\eta)
    -
    \frac{d_H}{2}\mathcal V_N(\eta).
\end{equation*}
The fixed-time coercivity argument depends on Cameron--Martin
embedding constants that vary with the length of the interval. For
example, when $H=1/2$, the Dirichlet--Poincar\'e inequality on
$[0,NT]$ gives
\begin{equation*}
    \|\eta\|_{L^2(0,NT)}^2
    \leq
    \left(\frac{NT}{\pi}\right)^2
    \|\eta'\|_{L^2(0,NT)}^2,
\end{equation*}
so the corresponding coercive lower bound deteriorates as
$N\to\infty$.

This degeneration alone, however, does not determine the sign of the
second variation. The tangential mode
\begin{equation*}
    \nu_t=\bar X_t'
\end{equation*}
belongs to the kernel of the linearized operator. Moreover,
differentiating \eqref{cond1} along the autonomous periodic orbit
yields
\begin{equation}
\label{eq:tangent-direct}
    \nabla^2(\nabla\!\cdot b)(\bar X_t)\bar X_t'=0.
\end{equation}
Thus the potential term also vanishes along the pure tangential mode.
The long-time behavior is therefore determined by the slow modulation
of this neutral mode and its interaction with the transversally
hyperbolic directions, rather than by the degeneration of the
Poincar\'e constant alone.

Our analysis uses Floquet coordinates to separate the tangential and
transverse components, together with Fourier methods to control the
fractional kinetic operator on $[0,NT]$. Since these arguments use both
endpoint conditions, we restrict the analysis to most probable
transition paths.

The results exhibit a genuine Hurst-dependent dichotomy. For
$H>1/2$, under an appropriate effective-potential condition, the
second variation loses positive definiteness for all sufficiently
large $N$. Consequently, the periodically extended deterministic
trajectory ceases to be a local minimizer. Since the resulting
negative direction also satisfies the free-endpoint constraint, local
minimality is destroyed in the evolution-path problem as well.

By contrast, for $1/4<H\leq1/2$, the second variation remains positive
definite over arbitrarily many periods under the corresponding
small-noise conditions. The resulting estimate is anisotropic: the
tangential component is controlled by the appropriate
Cameron--Martin or fractional kinetic norm, whereas the transverse
component is controlled in $L^2$.

This anisotropic estimate yields positive definiteness of the second
variation but not coercivity with respect to a single path-space norm.
In infinite dimensions, the Banach--Alaoglu theorem provides only weak
compactness of the closed unit ball, while a normalized sequence may
converge weakly to zero. Hence strict positivity need not yield a
uniform spectral gap, and therefore does not by itself imply local
minimality.

\subsection{Loss of local minimality for $H>1/2$}

In this subsection, we show that, under an effective-potential
condition, the second variation loses positive definiteness on
sufficiently long time intervals. Consequently, the periodically
extended deterministic trajectory $\bar X$ ceases to be a local
minimizer of $J_{\epsilon,N}$.

\begin{lemma}[Scaling of the weighted fractional kinetic energy]
\label{lem:scaling-weighted-fractional-derivative}
Let $f:[0,NT]\to\mathbb R^n$ and define
$f_N(\tau):=f(N\tau)$ for $\tau\in[0,T]$. Then
\begin{equation}
\label{eq:scaling-weighted-fractional-derivative}
\begin{aligned}
&\int_0^{NT}
\left|
    t^\alpha D_{0^+}^\alpha
    \bigl(t^{-\alpha}f(t)\bigr)
\right|^2\,dt
\\
&\qquad=
N^{1-2\alpha}
\int_0^T
\left|
    \tau^\alpha D_{0^+}^\alpha
    \bigl(\tau^{-\alpha}f_N(\tau)\bigr)
\right|^2\,d\tau,
\end{aligned}
\end{equation}
whenever either side is finite.
\end{lemma}

\begin{proof}
The homogeneity of the Riemann--Liouville derivative gives
\begin{equation*}
\left[
    t^\alpha D_{0^+}^\alpha
    \bigl(t^{-\alpha}f(t)\bigr)
\right](N\tau)
=
N^{-\alpha}
\tau^\alpha D_{0^+}^\alpha
\bigl(\tau^{-\alpha}f_N(\tau)\bigr).
\end{equation*}
The result follows by the change of variables $t=N\tau$.
\end{proof}

Thus, for slowly varying rescaled profiles, the kinetic contribution
is of order $N^{1-2\alpha}$, whereas the potential contribution is of
order $N$. Since $\alpha>0$, the potential term may dominate over long
time intervals.

Let
\begin{equation*}
    (\mathcal L\omega)_s
    :=
    \omega_s'
    -
    \nabla b(\bar X_s)\omega_s
\end{equation*}
be the linearized operator along $\bar X$. We seek periodic solutions
of
\begin{equation}
\label{eq:constant-response}
    \mathcal L\omega_s=\sigma_s q,
    \qquad
    \omega_0=\omega_T,
\end{equation}
where $q\in\mathbb R^n$ is independent of $s$. Such solutions will be
called constant-response modes.

Let $y$ be a nonzero periodic solution of
\begin{equation}
\label{eq:left-zero-mode}
    y_s'
    +
    \nabla b(\bar X_s)^Ty_s
    =
    0,
    \qquad
    y_0=y_T.
\end{equation}

\begin{lemma}[Fredholm characterization]
\label{lem:constant-response-fredholm}
Assume that $\bar X$ is transversally hyperbolic and that
\begin{equation}
\label{eq:zcr-nondegenerate}
    z_{\mathrm{CR}}
    :=
    \int_0^T\sigma_s^Ty_s\,ds
    \neq0.
\end{equation}
Define
\begin{equation}
\label{eq:ACR-def}
    \mathcal A_{\mathrm{CR}}
    :=
    \left\{
        q\in\mathbb R^n:
        z_{\mathrm{CR}}^Tq=0
    \right\}.
\end{equation}
Then \eqref{eq:constant-response} admits a periodic solution if and
only if $q\in\mathcal A_{\mathrm{CR}}$. For each such $q$, the solution
is unique modulo $\operatorname{span}\{\bar X'\}$.
\end{lemma}

\begin{proof}
By the Fredholm alternative, \eqref{eq:constant-response} is solvable
if and only if
\begin{equation*}
    \int_0^T y_s^T\sigma_s q\,ds
    =
    z_{\mathrm{CR}}^Tq
    =
    0.
\end{equation*}
Transversal hyperbolicity gives
\begin{equation*}
    \ker(\mathcal L_{\mathrm{per}})
    =
    \operatorname{span}\{\bar X'\},
\end{equation*}
which proves uniqueness modulo the tangential mode.
\end{proof}

The constant-response space is therefore defined by
\begin{equation}
\label{eq:WCR-quotient}
\begin{aligned}
\widetilde{\mathcal W}_{\mathrm{CR}}
:=
\Bigl\{
    [\omega]\in
    H^1_{\mathrm{per}}([0,T];\mathbb R^n)
    \big/
    \operatorname{span}\{\bar X'\}:\,
    \mathcal L\omega=\sigma q
    \text{ for some }
    q\in\mathcal A_{\mathrm{CR}}
\Bigr\}.
\end{aligned}
\end{equation}
Lemma~\ref{lem:constant-response-fredholm} implies that
$\widetilde{\mathcal W}_{\mathrm{CR}}$ is naturally isomorphic to
$\mathcal A_{\mathrm{CR}}$ and hence has dimension $n-1$.

For $[\omega]\in\widetilde{\mathcal W}_{\mathrm{CR}}$, define the
effective potential by
\begin{equation}
\label{eq:effective-quadratic-form}
    \mathcal V_{\mathrm{eff}}([\omega])
    :=
    -
    \int_0^T
    \omega_s^T
    \nabla^2(\nabla\!\cdot b)(\bar X_s)
    \omega_s\,ds.
\end{equation}
This definition is independent of the representative. Indeed,
\eqref{eq:tangent-direct} implies
\begin{equation*}
    \nabla^2(\nabla\!\cdot b)(\bar X_s)\bar X_s'=0,
\end{equation*}
so replacing $\omega$ by $\omega+c\bar X'$ does not change
\eqref{eq:effective-quadratic-form}.

We shall prove that if
\begin{equation}
\label{eq:positive-effective-potential}
    \mathcal V_{\mathrm{eff}}([\omega])>0
\end{equation}
for some
$[\omega]\in\widetilde{\mathcal W}_{\mathrm{CR}}$, then, for every
fixed $\epsilon>0$, the second variation of $J_{\epsilon,N}$ has a
negative direction for all sufficiently large $N$. The corresponding
perturbations are obtained by slowly modulating the constant-response
mode while imposing the endpoint conditions.

\begin{theorem}[Loss of local minimality on long time intervals for $H>1/2$]
\label{thm:H-greater-half-collapse}
Let $H>1/2$, and let $\bar X$ be a transversally hyperbolic
$T$-periodic solution of \eqref{nnoise} in $\mathbb R^n$, $n\geq2$.
Suppose that \eqref{cond1} holds and that the effective potential has
a positive direction:
\begin{equation*}
    \mathcal V_{\mathrm{eff}}([\omega^*])>0
\end{equation*}
for some
$[\omega^*]\in\widetilde{\mathcal W}_{\mathrm{CR}}$. Then, for every
$\epsilon>0$, there exists
$N_{\mathrm{crit}}(\epsilon)\in\mathbb N^+$ such that, whenever
$N>N_{\mathrm{crit}}(\epsilon)$, the periodic extension of $\bar X$ to
$[0,NT]$ is not a local minimizer of $J_{\epsilon,N}$ in
\begin{equation*}
    x_0+
    \mathcal H_H^\sigma\bigl([0,NT];\mathbb R^n\bigr).
\end{equation*}
\end{theorem}

\begin{proof}
Set $\alpha=H-\frac12\in(0,\frac12)$. For an admissible perturbation
$\eta$, write
\begin{equation}
\label{eq:IHN-def-H-greater-half}
    \mathcal I_H^N(\eta)
    :=
    \int_0^{NT}
    \left|
        s^\alpha D_{0^+}^{\alpha}
        \left[
            s^{-\alpha}\sigma_s^{-1}
            \bigl(
                \eta_s'
                -
                \nabla b(\bar X_s)\eta_s
            \bigr)
        \right]
    \right|^2\,ds
\end{equation}
and
\begin{equation}
\label{eq:VN-def-H-greater-half}
    \mathcal V^N(\eta)
    :=
    -
    \int_0^{NT}
    \eta_s^T
    \nabla^2(\nabla\!\cdot b)(\bar X_s)
    \eta_s\,ds.
\end{equation}
Then
\begin{equation*}
    \delta^2J_{\epsilon,N}(\bar X)[\eta,\eta]
    =
    \frac1\epsilon\mathcal I_H^N(\eta)
    -
    \frac{d_H}{2}\mathcal V^N(\eta).
\end{equation*}

Choose a representative $\omega^*$ of $[\omega^*]$ and set
\begin{equation}
\label{eq:positive-Veff}
    C_V
    :=
    \mathcal V_{\mathrm{eff}}([\omega^*])
    =
    -
    \int_0^T
    (\omega_s^*)^T
    \nabla^2(\nabla\!\cdot b)(\bar X_s)
    \omega_s^*\,ds
    >
    0.
\end{equation}
Since $\omega^*$ is a constant-response mode, there exists a constant
$q^*\in\mathbb R^n$ such that
\begin{equation}
\label{eq:constant-response-selected}
    \sigma_s^{-1}
    \bigl(
        (\omega_s^*)'
        -
        \nabla b(\bar X_s)\omega_s^*
    \bigr)
    =
    q^*.
\end{equation}

Choose a nonzero $\chi\in C^\infty([0,T])$ satisfying
\begin{equation*}
    \chi(0)=\chi(T)=\chi'(0)=\chi'(T)=0
\end{equation*}
and define
\begin{equation}
\label{eq:test-variation-etaN}
    \eta_N(s)
    :=
    \chi\left(\frac{s}{N}\right)\omega_s^*,
    \qquad
    s\in[0,NT].
\end{equation}
Then $\eta_N(0)=\eta_N(NT)=0$, so $\eta_N$ is admissible in both the
fixed-endpoint and free-endpoint perturbation spaces.

We first estimate the potential term. Since
\begin{equation*}
    F(s)
    :=
    (\omega_s^*)^T
    \nabla^2(\nabla\!\cdot b)(\bar X_s)
    \omega_s^*
\end{equation*}
is $T$-periodic and
$\int_0^T F(s)\,ds=-C_V$, periodic averaging gives
\begin{align*}
    \mathcal V^N(\eta_N)
    &=
    -
    \int_0^{NT}
    \chi^2\left(\frac{s}{N}\right)F(s)\,ds
    \\
    &=
    \frac{C_V}{T}
    \left(
        \int_0^T\chi^2(u)\,du
    \right)N
    +
    o(N).
\end{align*}
Hence, for some $C>0$ and all sufficiently large $N$,
\begin{equation}
\label{eq:potential-linear-positive}
    \mathcal V^N(\eta_N)\geq CN.
\end{equation}

For the kinetic term, \eqref{eq:constant-response-selected} and
\eqref{eq:test-variation-etaN} yield
\begin{equation*}
\begin{aligned}
&\sigma_s^{-1}
\bigl(
    \eta_N'(s)-\nabla b(\bar X_s)\eta_N(s)
\bigr)
\\
&\qquad=
\chi\left(\frac{s}{N}\right)q^*
+
\frac1N
\chi'\left(\frac{s}{N}\right)
\sigma_s^{-1}\omega_s^*.
\end{aligned}
\end{equation*}
The scaling estimate
\eqref{eq:scaling-weighted-fractional-derivative} gives
\begin{equation*}
\begin{aligned}
&\int_0^{NT}
\left|
s^\alpha D_{0^+}^{\alpha}s^{-\alpha}
\left[
    \chi\left(\frac{s}{N}\right)q^*
\right]
\right|^2\,ds
\leq
CN^{1-2\alpha}.
\end{aligned}
\end{equation*}

Set $P_*(s):=\sigma_s^{-1}\omega_s^*$. For the remaining term, scaling
gives
\begin{align*}
&\frac1{N^2}
\int_0^{NT}
\left|
s^\alpha D_{0^+}^{\alpha}s^{-\alpha}
\left[
    \chi'\left(\frac{s}{N}\right)P_*(s)
\right]
\right|^2\,ds
\\
&\qquad=
N^{-1-2\alpha}
\int_0^T
\tau^{2\alpha}
\left|
D_{0^+}^{\alpha}
\left[
    \tau^{-\alpha}\chi'(\tau)P_*(N\tau)
\right]
\right|^2\,d\tau.
\end{align*}
Because $\chi'(0)=0$, $P_*$ is smooth and periodic, and
$\alpha<\frac12$, the fractional integral estimate
$D_{0^+}^{\alpha}=I_{0^+}^{1-\alpha}\frac{d}{d\tau}$ gives
\begin{equation*}
    \int_0^T
    \tau^{2\alpha}
    \left|
    D_{0^+}^{\alpha}
    \left[
        \tau^{-\alpha}\chi'(\tau)P_*(N\tau)
    \right]
    \right|^2\,d\tau
    \leq
    CN^2.
\end{equation*}
Consequently,
\begin{equation}
\label{eq:kinetic-upper-H-greater-half}
    \mathcal I_H^N(\eta_N)
    \leq
    CN^{1-2\alpha}.
\end{equation}

Combining \eqref{eq:potential-linear-positive} and
\eqref{eq:kinetic-upper-H-greater-half}, we obtain
\begin{equation}
\label{eq:second-var-final-estimate}
    \delta^2J_{\epsilon,N}(\bar X)[\eta_N,\eta_N]
    \leq
    \frac{C}{\epsilon}N^{1-2\alpha}
    -
    CN.
\end{equation}
Since $\alpha>0$, the first term grows sublinearly, whereas the second
is negative and grows linearly. Thus
\begin{equation*}
    \delta^2J_{\epsilon,N}(\bar X)[\eta_N,\eta_N]<0
\end{equation*}
for all sufficiently large $N$. Therefore, $\bar X$ is not a local
minimizer of $J_{\epsilon,N}$ in
$x_0+\mathcal H_H^\sigma([0,NT];\mathbb R^n)$.
\end{proof}

\begin{remark}
The loss of local minimality is caused by the different long-time
scalings of the two parts of the second variation: the fractional
kinetic contribution is of order $N^{1-2\alpha}$, whereas the
destabilizing potential contribution is of order $N$. Analytically,
this reflects the homogeneity of the fractional Cameron--Martin
operator; probabilistically, it is consistent with the positive
long-range dependence of fractional Brownian motion for $H>1/2$.
\end{remark}

\subsection{Anisotropic positive definiteness in the Brownian regime
$H=1/2$}

We next consider the Brownian case $H=1/2$. The tangential zero mode
prevents uniform coercivity in the full Cameron--Martin norm as
$N\to\infty$. Nevertheless, the potential term vanishes in the
tangential direction, while transverse hyperbolicity provides uniform
control of the transverse component. Consequently, for sufficiently
small noise intensity, the second variation remains positive definite
for every $N$.

The resulting estimate is anisotropic: the tangential coefficient is
controlled through its homogeneous derivative norm, whereas the
transverse component is controlled in an $H^1$-type norm.

\begin{theorem}[Uniform anisotropic positive definiteness for $H=1/2$]
\label{thm:persistence-H-half}
Let $H=1/2$, and let $\bar X$ be a transversally hyperbolic
$T$-periodic solution of \eqref{nnoise} in $\mathbb R^n$, $n\geq2$.
Suppose that \eqref{cond1} holds. Then there exists $\epsilon_*>0$,
independent of $N$, such that, for every $N\in\mathbb N^+$ and every
$0<\epsilon<\epsilon_*$,
\begin{equation*}
    \delta^2J_{\epsilon,N}(\bar X)[\eta,\eta]>0
\end{equation*}
for every nonzero
\begin{equation*}
    \eta\in
    \mathcal H_{1/2,0}^\sigma
    \bigl([0,NT];\mathbb R^n\bigr).
\end{equation*}
More precisely, in the Floquet coordinates
\begin{equation*}
    P(s)^{-1}\eta_s
    =
    \begin{pmatrix}
        z_s^0\\
        z_s^\perp
    \end{pmatrix},
\end{equation*}
there exists $c_\epsilon>0$, independent of $N$, such that
\begin{equation}
\label{eq:second-var-lower-H-half}
    \delta^2J_{\epsilon,N}(\bar X)[\eta,\eta]
    \geq
    c_\epsilon
    \left(
        \|(z^0)'\|_{L^2(0,NT)}^2
        +
        \|z^\perp\|_{H^1(0,NT)}^2
    \right).
\end{equation}
\end{theorem}

\begin{proof}
Let
\begin{equation*}
    z_s
    :=
    P(s)^{-1}\eta_s
    =
    \begin{pmatrix}
        z_s^0\\
        z_s^\perp
    \end{pmatrix},
\end{equation*}
where $P$ is a $T$-periodic Floquet matrix whose first column is
$\bar X'$. In these coordinates,
\begin{equation*}
    \mathcal L\eta_s
    =
    P(s)(z_s'-Rz_s),
    \qquad
    R=
    \begin{pmatrix}
        0&0\\
        0&R_\perp
    \end{pmatrix},
\end{equation*}
and transverse hyperbolicity implies
\begin{equation*}
    \sigma(R_\perp)\cap i\mathbb R=\varnothing.
\end{equation*}

For $H=1/2$, the second variation is
\begin{equation}
\label{eq:second-variation-H-half}
    \delta^2J_{\epsilon,N}(\bar X)[\eta,\eta]
    =
    \frac1\epsilon\mathcal I_{1/2,N}(\eta)
    -
    \frac12\mathcal V_N(\eta),
\end{equation}
where
\begin{equation}
\label{eq:kinetic-H-half}
    \mathcal I_{1/2,N}(\eta)
    :=
    \int_0^{NT}
    \left|
        \sigma_s^{-1}
        \bigl(
            \eta_s'-\nabla b(\bar X_s)\eta_s
        \bigr)
    \right|^2\,ds
\end{equation}
and
\begin{equation}
\label{eq:potential-H-half}
    \mathcal V_N(\eta)
    :=
    -
    \int_0^{NT}
    \eta_s^T
    \nabla^2(\nabla\!\cdot b)(\bar X_s)
    \eta_s\,ds.
\end{equation}

The boundedness and uniform invertibility of $\sigma_s$ and $P(s)$
give
\begin{equation}
\label{eq:kinetic-lower-bound-H-half}
\begin{aligned}
    \mathcal I_{1/2,N}(\eta)
    \geq
    C
    \left(
        \|(z^0)'\|_{L^2(0,NT)}^2
        +
        \|(z^\perp)'-R_\perp z^\perp\|_{L^2(0,NT)}^2
    \right),
\end{aligned}
\end{equation}
where $C>0$ is independent of $N$.

Since $\sigma(R_\perp)\cap i\mathbb R=\varnothing$, there exists
$C>0$ such that
\begin{equation*}
    |(i\xi I-R_\perp)y|^2
    \geq
    C(1+\xi^2)|y|^2,
    \qquad
    \xi\in\mathbb R,\quad
    y\in\mathbb C^{n-1}.
\end{equation*}
Moreover, the endpoint conditions imply
$z^\perp(0)=z^\perp(NT)=0$, so $z^\perp$ can be viewed as an
$H^1$-periodic function on $[0,NT]$. The Fourier multiplier estimate
therefore yields
\begin{equation}
\label{eq:fourier-transverse-bound}
    \|(z^\perp)'-R_\perp z^\perp\|_{L^2(0,NT)}^2
    \geq
    C\|z^\perp\|_{H^1(0,NT)}^2,
\end{equation}
with $C$ independent of $N$. Combining
\eqref{eq:kinetic-lower-bound-H-half} and
\eqref{eq:fourier-transverse-bound}, we obtain
\begin{equation}
\label{eq:I-lower-bound-final}
    \mathcal I_{1/2,N}(\eta)
    \geq
    C
    \left(
        \|(z^0)'\|_{L^2(0,NT)}^2
        +
        \|z^\perp\|_{H^1(0,NT)}^2
    \right).
\end{equation}

Set
\begin{equation*}
    M(s)
    :=
    \nabla^2(\nabla\!\cdot b)(\bar X_s).
\end{equation*}
By \eqref{eq:tangent-direct},
\begin{equation*}
    M(s)\bar X_s'=0.
\end{equation*}
Since the first column of $P(s)$ is $\bar X_s'$, the tangential and
mixed terms vanish from the potential energy. Hence
\begin{equation}
\label{eq:potential-bound-H-half}
    |\mathcal V_N(\eta)|
    \leq
    C\|z^\perp\|_{L^2(0,NT)}^2
    \leq
    C\|z^\perp\|_{H^1(0,NT)}^2,
\end{equation}
where $C$ is independent of $N$.

Substituting \eqref{eq:I-lower-bound-final} and
\eqref{eq:potential-bound-H-half} into
\eqref{eq:second-variation-H-half} gives
\begin{align*}
    \delta^2J_{\epsilon,N}(\bar X)[\eta,\eta]
    \geq{}&
    \frac{C}{\epsilon}
    \|(z^0)'\|_{L^2(0,NT)}^2
    \\
    &+
    \left(
        \frac{C}{\epsilon}-C
    \right)
    \|z^\perp\|_{H^1(0,NT)}^2.
\end{align*}
Choosing $\epsilon_*>0$ sufficiently small proves
\eqref{eq:second-var-lower-H-half}, with $c_\epsilon>0$ independent of
$N$.

If the second variation vanishes, then
$z^\perp=0$ and $(z^0)'=0$. Since $z^0(0)=0$, it follows that
$z^0=0$, and hence $\eta=0$. Therefore, the second variation is
positive definite.
\end{proof}

\begin{remark}
The estimate is uniform in $N$ but anisotropic. Transverse
hyperbolicity gives uniform $H^1$ control of $z^\perp$, whereas the
tangential kinetic term controls only $(z^0)'$. Although the endpoint
conditions imply
\begin{equation*}
    \|z^0\|_{L^2(0,NT)}
    \leq
    \frac{NT}{\pi}
    \|(z^0)'\|_{L^2(0,NT)},
\end{equation*}
the constant grows with $N$. Hence
\eqref{eq:second-var-lower-H-half} does not provide uniform coercivity
with respect to the full Cameron--Martin norm as $N\to\infty$.
\end{remark}
\subsection{Anisotropic Positive Definiteness for $1/4<H<1/2$}

We finally consider the regime $1/4<H<1/2$.
The main difficulty is the weighted fractional integral in the kinetic
term. We first establish a uniform lower bound for this operator and then
combine it with the Floquet decomposition of the linearized equation.

For $L>0$, define
\begin{equation*}
    (T_\alpha f)(s)
    :=
    s^{-\alpha}I_{0^+}^{\alpha}(s^\alpha f)(s),
    \qquad
    0<s<L.
\end{equation*}

\begin{lemma}[Weighted right-sided fractional derivative estimate]
\label{lem:weighted-right-fractional-derivative}
Let $0<\alpha<1/2$. There exists $C_{\alpha,T}>0$ such that
\begin{equation}
\label{eq:weighted-right-fractional-derivative}
    \left\|
        s^\alpha
        D_{T^-}^{\alpha}
        \bigl(s^{-\alpha}\varphi\bigr)
    \right\|_{L^2(0,T)}
    \leq
    C_{\alpha,T}
    \|\varphi\|_{H^\alpha(0,T)}
\end{equation}
for every $\varphi\in H^\alpha(0,T;\mathbb R^n)$.
\end{lemma}
    
\begin{proof}
It suffices to consider
$\varphi\in C_c^\infty(0,T;\mathbb R^n)$, since
$C_c^\infty(0,T)$ is dense in $H^\alpha(0,T)$ for $\alpha<1/2$.
By the Marchaud representation of the right-sided
Riemann--Liouville derivative \cite{Samko1993},
\begin{equation}
\label{eq:weighted-right-derivative-decomposition}
    s^\alpha
    D_{T^-}^{\alpha}
    \bigl(s^{-\alpha}\varphi\bigr)
    =
    D_{T^-}^{\alpha}\varphi
    +
    R_\alpha\varphi,
\end{equation}
where
\begin{equation*}
    R_\alpha\varphi(s)
    =
    \frac{\alpha}{\Gamma(1-\alpha)}
    \int_s^T
    \frac{1-(s/r)^\alpha}
         {(r-s)^{1+\alpha}}
    \varphi(r)\,dr.
\end{equation*}
The standard Riemann--Liouville characterization of fractional Sobolev
spaces gives
\begin{equation}
\label{eq:right-derivative-Halpha-bound}
    \|D_{T^-}^{\alpha}\varphi\|_{L^2(0,T)}
    \leq
    C_{\alpha,T}
    \|\varphi\|_{H^\alpha(0,T)}.
\end{equation}

Setting $r=su$ and applying Minkowski's inequality, we obtain
\begin{align*}
    \|R_\alpha\varphi\|_{L^2(0,T)}
    &\leq
    C_\alpha
    \left[
        \int_1^\infty
        \frac{1-u^{-\alpha}}
             {(u-1)^{1+\alpha}}
        u^{\alpha-\frac12}\,du
    \right]
    \|s^{-\alpha}\varphi\|_{L^2(0,T)}
    \\
    &\leq
    C_\alpha
    \|s^{-\alpha}\varphi\|_{L^2(0,T)},
\end{align*}
since the integral in brackets is finite. Moreover,
\cite[Eq.~(17)]{Dyda2004FractionalHardy}, applied with
$D=(0,T)$, $p=2$, and parameter $2\alpha$, gives
\begin{equation*}
    \|s^{-\alpha}\varphi\|_{L^2(0,T)}
    \leq
    C_{\alpha,T}
    \|\varphi\|_{H^\alpha(0,T)}.
\end{equation*}
Therefore,
\begin{equation*}
    \|R_\alpha\varphi\|_{L^2(0,T)}
    \leq
    C_{\alpha,T}
    \|\varphi\|_{H^\alpha(0,T)}.
\end{equation*}
Combining this estimate with
\eqref{eq:weighted-right-derivative-decomposition} and
\eqref{eq:right-derivative-Halpha-bound} proves
\begin{equation*}
    \left\|
        s^\alpha
        D_{T^-}^{\alpha}
        \bigl(s^{-\alpha}\varphi\bigr)
    \right\|_{L^2(0,T)}
    \leq
    C_{\alpha,T}
    \|\varphi\|_{H^\alpha(0,T)}.
\end{equation*}
The general case follows by density.
\end{proof}

\begin{lemma}[Uniform fractional lower bound]
\label{lem:uniform-fractional-lower-bound}
There exists a constant $c_{\alpha,T}>0$, independent of $N$, such that
\begin{equation}
\label{eq:uniform-fractional-lower-bound}
    \|T_\alpha f\|_{L^2(0,NT)}
    \geq
    c_{\alpha,T}
    \|f\|_{H^{-\alpha}(\mathbb T_{NT})}
\end{equation}
for every $N\in\mathbb N^+$ and every
$f\in L^2(0,NT;\mathbb R^n)$.
\end{lemma}

\begin{proof}
We first prove the estimate on $(0,T)$. Set $g=T_\alpha f$. Then
\begin{equation*}
    f
    =
    s^{-\alpha}
    D_{0^+}^{\alpha}(s^\alpha g)
\end{equation*}
in the distributional sense. For
$\varphi\in H^\alpha(\mathbb T_T;\mathbb R^n)$, fractional integration
by parts gives
\begin{align*}
    \int_0^T f(s)^T\varphi(s)\,ds
    &=
    \int_0^T
    g(s)^T
    s^\alpha
    D_{T^-}^{\alpha}
    \bigl(s^{-\alpha}\varphi(s)\bigr)\,ds.
\end{align*}
By Lemma~\ref{lem:weighted-right-fractional-derivative},
\begin{equation*}
    \left|
        \int_0^T f(s)^T\varphi(s)\,ds
    \right|
    \leq
    C_{\alpha,T}
    \|T_\alpha f\|_{L^2(0,T)}
    \|\varphi\|_{H^\alpha(\mathbb T_T)}.
\end{equation*}
Taking the supremum over nonzero $\varphi$ gives
\begin{equation}
\label{eq:fractional-lower-fixed}
    \|T_\alpha f\|_{L^2(0,T)}
    \geq
    c_{\alpha,T}
    \|f\|_{H^{-\alpha}(\mathbb T_T)}.
\end{equation}

For $f\in L^2(0,NT)$, define
\begin{equation*}
    f_N(\tau):=f(N\tau),
    \qquad
    0<\tau<T.
\end{equation*}
The scaling property of $T_\alpha$ gives
\begin{equation}
\label{eq:Talpha-scaling}
    \|T_\alpha f\|_{L^2(0,NT)}
    =
    N^{\alpha+\frac12}
    \|T_\alpha f_N\|_{L^2(0,T)}.
\end{equation}
If
\begin{equation*}
    f(s)
    =
    \sum_{k\in\mathbb Z}
    \widehat f_k e^{2\pi iks/(NT)},
\end{equation*}
then
\begin{equation*}
    f_N(\tau)
    =
    \sum_{k\in\mathbb Z}
    \widehat f_k e^{2\pi ik\tau/T}.
\end{equation*}
Here and below, we use the notation
\begin{equation*}
    \langle\xi\rangle
    :=
    \bigl(1+|\xi|^2\bigr)^{1/2}.
\end{equation*}
Using
\begin{equation*}
    \left\langle\frac{\xi}{N}\right\rangle^{-\alpha}
    \leq
    N^\alpha\langle\xi\rangle^{-\alpha},
\end{equation*}
we obtain
\begin{equation}
\label{eq:negative-Sobolev-scaling}
    \|f\|_{H^{-\alpha}(\mathbb T_{NT})}
    \leq
    N^{\alpha+\frac12}
    \|f_N\|_{H^{-\alpha}(\mathbb T_T)}.
\end{equation}
Combining
\eqref{eq:fractional-lower-fixed}--%
\eqref{eq:negative-Sobolev-scaling} proves the result.
\end{proof}

\begin{lemma}[Uniform periodic multiplier estimate]
\label{lem:uniform-periodic-multiplier}
Let $r>\alpha+1/2$ and suppose that
\begin{equation*}
    A,A^{-1}
    \in
    H^r
    \bigl(
        \mathbb T_T;\mathbb R^{n\times n}
    \bigr).
\end{equation*}
Then there exists $c_A>0$, independent of $N$, such that
\begin{equation}
\label{eq:uniform-periodic-multiplier}
    \|Af\|_{H^{-\alpha}(\mathbb T_{NT})}
    \geq
    c_A
    \|f\|_{H^{-\alpha}(\mathbb T_{NT})}
\end{equation}
for every $N\in\mathbb N^+$ and every
$f\in H^{-\alpha}(\mathbb T_{NT};\mathbb R^n)$.
\end{lemma}

\begin{proof}
Write
\begin{equation*}
    A(s)
    =
    \sum_{m\in\mathbb Z}
    A_m e^{2\pi ims/T}.
\end{equation*}
Multiplication by the $m$-th Fourier mode shifts the physical frequency
by $2\pi m/T$, independently of $N$. Hence
\begin{equation*}
    \|Af\|_{H^{-\alpha}(\mathbb T_{NT})}
    \leq
    C
    \left(
        \sum_{m\in\mathbb Z}
        \|A_m\|
        \left\langle\frac{2\pi m}{T}\right\rangle^\alpha
    \right)
    \|f\|_{H^{-\alpha}(\mathbb T_{NT})}.
\end{equation*}
Since $r>\alpha+1/2$, the sum is bounded by
$C\|A\|_{H^r(\mathbb T_T)}$. Applying the same estimate to $A^{-1}$
gives \eqref{eq:uniform-periodic-multiplier}.
\end{proof}

\begin{theorem}[Uniform anisotropic positive definiteness for
$1/4<H<1/2$]
\label{thm:persistence-H-lesshalf}
Let $1/4<H<1/2$, and let $\bar X$ be a transversally hyperbolic
$T$-periodic solution of \eqref{nnoise} in $\mathbb R^n$, $n\geq2$,
with
\begin{equation*}
    \bar X_0=\bar X_T=x_0.
\end{equation*}
Assume that condition \eqref{cond1} holds. Let $P$ be a $T$-periodic
Floquet matrix whose first column is $\bar X'$, and suppose that
\begin{equation*}
    \sigma^{-1}P,\,
    P^{-1}\sigma
    \in
    H^r
    \bigl(
        \mathbb T_T;\mathbb R^{n\times n}
    \bigr)
\end{equation*}
for some $r>\alpha+1/2$. Then there exists $\epsilon_*>0$, independent
of $N$, such that, for every $N\in\mathbb N^+$ and every
$0<\epsilon<\epsilon_*$,
\begin{equation*}
    \delta^2J_{\epsilon,N}(\bar X)[\eta,\eta]>0
\end{equation*}
for every nonzero
\begin{equation*}
    \eta\in
    \mathcal H_{H,0}^{\sigma}
    \bigl([0,NT];\mathbb R^n\bigr).
\end{equation*}
\end{theorem}

\begin{proof}
We first consider smooth admissible perturbations; the general case
follows by density. Write
\begin{equation*}
    \eta_s=P(s)z_s,
    \qquad
    z_s=
    \begin{pmatrix}
        z_s^0\\
        z_s^\perp
    \end{pmatrix}.
\end{equation*}
The Floquet representation gives
\begin{equation}
\label{eq:floquet-linearized-H-lesshalf}
    \eta_s'
    -
    \nabla b(\bar X_s)\eta_s
    =
    P(s)(z_s'-Rz_s),
\end{equation}
where
\begin{equation*}
    R=
    \begin{pmatrix}
        0&0\\
        0&R_\perp
    \end{pmatrix},
    \qquad
    \sigma(R_\perp)\cap i\mathbb R=\varnothing.
\end{equation*}

Set
\begin{equation*}
    M(s):=\sigma_s^{-1}P(s).
\end{equation*}
The kinetic part of the second variation is
\begin{equation*}
    \mathcal I_{H,N}(\eta)
    =
    \left\|
        T_\alpha
        \bigl(
            M(z'-Rz)
        \bigr)
    \right\|_{L^2(0,NT)}^2.
\end{equation*}
By Lemmas~\ref{lem:uniform-fractional-lower-bound} and
\ref{lem:uniform-periodic-multiplier},
\begin{equation}
\label{eq:kinetic-negative-Sobolev-H-lesshalf}
    \mathcal I_{H,N}(\eta)
    \geq
    C
    \|z'-Rz\|_{H^{-\alpha}(\mathbb T_{NT})}^2,
\end{equation}
where $C>0$ is independent of $N$.

The block structure of $R$ gives
\begin{align*}
    \|z'-Rz\|_{H^{-\alpha}}^2
    &=
    \|(z^0)'\|_{H^{-\alpha}}^2
    +
    \|(z^\perp)'-R_\perp z^\perp\|_{H^{-\alpha}}^2.
\end{align*}
Transverse hyperbolicity implies that
\begin{equation*}
    |(i\xi I-R_\perp)y|
    \geq
    c_\perp\langle\xi\rangle|y|,
    \qquad
    \xi\in\mathbb R.
\end{equation*}
Applying this estimate to each Fourier mode on $\mathbb T_{NT}$ yields
\begin{equation*}
    \|(z^\perp)'-R_\perp z^\perp\|_{H^{-\alpha}}^2
    \geq
    c_\perp^2
    \|z^\perp\|_{H^{1-\alpha}}^2.
\end{equation*}
Consequently,
\begin{equation}
\label{eq:anisotropic-kinetic-H-lesshalf}
\begin{aligned}
    \mathcal I_{H,N}(\eta)
    \geq{}&
    C_\parallel
    \|(z^0)'\|_{H^{-\alpha}(\mathbb T_{NT})}^2
    \\
    &+
    C_\perp
    \|z^\perp\|_{H^{1-\alpha}(\mathbb T_{NT})}^2,
\end{aligned}
\end{equation}
with constants independent of $N$.

Differentiating condition \eqref{cond1} along the periodic orbit gives
\begin{equation*}
    \nabla^2(\nabla\!\cdot b)(\bar X_s)\bar X_s'=0.
\end{equation*}
Thus the tangential and tangent--transverse contributions to the
potential term vanish. Since the remaining coefficient is continuous and
$T$-periodic,
\begin{equation}
\label{eq:potential-H-lesshalf}
    |\mathcal V_N(\eta)|
    \leq
    C_V
    \|z^\perp\|_{L^2(0,NT)}^2,
\end{equation}
where $C_V$ is independent of $N$.

Using
\begin{equation*}
    \delta^2J_{\epsilon,N}(\bar X)[\eta,\eta]
    =
    \frac1\epsilon\mathcal I_{H,N}(\eta)
    -
    \frac{d_H}{2}\mathcal V_N(\eta),
\end{equation*}
we obtain
\begin{align}
    \delta^2J_{\epsilon,N}(\bar X)[\eta,\eta]
    &\geq
    \frac{C_\parallel}{\epsilon}
    \|(z^0)'\|_{H^{-\alpha}(\mathbb T_{NT})}^2
    \notag\\
    &\quad+
    \left(
        \frac{C_\perp}{\epsilon}
        -
        \frac{d_HC_V}{2}
    \right)
    \|z^\perp\|_{L^2(0,NT)}^2.
    \label{eq:second-variation-H-lesshalf-final}
\end{align}

Choose
\begin{equation*}
    0<\epsilon_*
    <
    \frac{2C_\perp}{d_HC_V},
\end{equation*}
with the convention that $\epsilon_*=+\infty$ if $C_V=0$. Then the
right-hand side of
\eqref{eq:second-variation-H-lesshalf-final} is nonnegative.

If the second variation vanishes, then
\begin{equation*}
    z^\perp=0,
    \qquad
    (z^0)'=0.
\end{equation*}
Since $\eta(0)=\eta(NT)=0$ and $P$ is invertible, we have
$z(0)=z(NT)=0$. Hence $z^0=0$, and therefore $\eta=0$. Thus the second
variation is positive definite for every $N$, with a noise threshold
independent of $N$.
\end{proof}

\begin{remark}
The estimate \eqref{eq:second-variation-H-lesshalf-final} is anisotropic:
the tangential component is controlled through
$\|(z^0)'\|_{H^{-\alpha}}$, whereas the transverse component is controlled
in $H^{1-\alpha}$ and hence in $L^2$. This proves uniform positive
definiteness of the second variation, but does not by itself imply uniform
local minimality on all intervals $[0,NT]$.
\end{remark}

    \section{KAM Tori as Most Probable Evolution Paths under Fractional Noise}

The persistence of invariant tori under small deterministic
perturbations is a central conclusion of classical KAM theory; see
\cite{Kolmogorov1954,Arnold1963,Moser1962}. A variational notion of
KAM persistence for Hamiltonian systems driven by Brownian noise was
introduced in \cite{xinze2026most}, where trajectories on the
deformed deterministic tori were characterized as most probable
evolution paths.

As an application of the divergence-free case established above, we
extend this result to time-dependent fractional noise. No new KAM
iteration is required. The deterministic KAM theorem first produces
deformed invariant tori carrying quasi-periodic trajectories, and the
fractional Onsager--Machlup theory then shows that these trajectories
are most probable evolution paths. Since the Hamiltonian drift is
divergence-free, this variational conclusion holds for every noise
intensity, although the deterministic perturbation must remain within
the KAM regime.

Let $D\subset\mathbb R^n$ be a nonempty bounded domain and consider
the nearly integrable Hamiltonian
\begin{equation}
\label{eq:nearly-integrable-Hamiltonian}
    \mathscr H(I,\theta)
    =
    H_0(I)+P(I,\theta),
    \qquad
    (I,\theta)\in D\times\mathbb T^n.
\end{equation}
The corresponding fractional stochastic Hamiltonian system is
\begin{equation}
\label{eq:fractional-stochastic-KAM-system}
\left\{
\begin{aligned}
    d\theta_t
    &=
    \left[
        \omega(I_t)
        +
        \partial_I P(I_t,\theta_t)
    \right]dt
    +
    \sqrt{\epsilon}\,
    \sigma_\theta(t)\,dB_{\theta,t}^{H},
    \\
    dI_t
    &=
    -
    \partial_\theta P(I_t,\theta_t)\,dt
    +
    \sqrt{\epsilon}\,
    \sigma_I(t)\,dB_{I,t}^{H},
\end{aligned}
\right.
\end{equation}
where
\begin{equation*}
    \omega(I):=\partial_IH_0(I),
    \qquad
    \frac14<H<1.
\end{equation*}
Here $B_\theta^H$ and $B_I^H$ are independent $n$-dimensional
fractional Brownian motions with the same Hurst parameter, and
$\epsilon>0$ is the noise intensity. The angle equation is understood
on the universal covering space $\mathbb R^n$ and then projected onto
$\mathbb T^n$ modulo $2\pi\mathbb Z^n$.

With the state variables ordered as $(\theta,I)$, the deterministic
drift is
\begin{equation}
\label{eq:KAM-Hamiltonian-vector-field}
    b_{\mathscr H}(\theta,I)
    :=
    \begin{pmatrix}
        \partial_I\mathscr H(I,\theta)
        \\
        -\partial_\theta\mathscr H(I,\theta)
    \end{pmatrix}
    =
    \begin{pmatrix}
        \omega(I)+\partial_I P(I,\theta)
        \\
        -\partial_\theta P(I,\theta)
    \end{pmatrix}.
\end{equation}
By equality of mixed partial derivatives,
\begin{equation}
\label{eq:KAM-divergence-free}
\begin{aligned}
    \nabla_{(\theta,I)}\!\cdot b_{\mathscr H}
    &=
    \nabla_\theta\!\cdot
    \bigl(\partial_I\mathscr H\bigr)
    -
    \nabla_I\!\cdot
    \bigl(\partial_\theta\mathscr H\bigr)
    \\
    &=0.
\end{aligned}
\end{equation}
Hence the divergence correction in the fractional
Onsager--Machlup functional vanishes, and the functional reduces to
its nonnegative kinetic part. Its minimum is therefore attained
precisely when the Cameron--Martin residual vanishes. By
Corollary~\ref{cor: spatially constant divergence}, every
deterministic trajectory of
\begin{equation}
\label{eq:deterministic-nearly-integrable-KAM-system}
\left\{
\begin{aligned}
    \dot\theta_t
    &=
    \omega(I_t)+\partial_I P(I_t,\theta_t),
    \\
    \dot I_t
    &=
    -\partial_\theta P(I_t,\theta_t)
\end{aligned}
\right.
\end{equation}
is consequently a most probable evolution path of
\eqref{eq:fractional-stochastic-KAM-system}.

\begin{theorem}[KAM tori in the most probable evolution-path sense]
\label{thm:fractional-most-probable-KAM-tori}
Let $1/4<H<1$ and $n\geq2$, and consider
\eqref{eq:fractional-stochastic-KAM-system}. Assume that
$b_{\mathscr H}$ and the diffusion coefficients satisfy
Assumption~\ref{ass:A}.

Let $\tau>n-1$, set $\nu:=\tau+1$, and suppose that the following
conditions hold:
\begin{enumerate}
    \item There exists $l>2\nu>2n$ such that
    \begin{equation*}
        H_0\in C^l(D),
        \qquad
        P\in C^l(D\times\mathbb T^n).
    \end{equation*}

    \item The integrable Hamiltonian satisfies
    \begin{equation*}
        \det\partial_I^2H_0(I)\neq0,
        \qquad
        I\in D,
    \end{equation*}
    and
    \begin{equation*}
        \sup_{I\in D}
        \left|
            \bigl(\partial_I^2H_0(I)\bigr)^{-1}
        \right|
        <\infty.
    \end{equation*}

    \item For some $\gamma\in(0,1)$, define
    \begin{equation*}
        \Delta_\gamma^\tau
        :=
        \left\{
            \varpi\in\mathbb R^n:
            |\varpi\cdot k|
            \geq
            \frac{\gamma}{|k|_1^\tau},
            \quad
            k\in\mathbb Z^n\setminus\{0\}
        \right\},
    \end{equation*}
    where $|k|_1:=\sum_{j=1}^n|k_j|$. Set
    \begin{equation*}
        \gamma_*:=\gamma^{1/(l-2\nu)}
    \end{equation*}
    and
    \begin{equation*}
        D_\gamma
        :=
        \left\{
            I\in D:
            B_{\gamma_*}(I)\subset D,\;
            \omega(I)\in\Delta_\gamma^\tau
        \right\}.
    \end{equation*}

    \item The deterministic perturbation satisfies
    \begin{equation*}
        \|P\|_{C^l(D\times\mathbb T^n)}
        \leq
        \eta_*,
    \end{equation*}
    where $\eta_*>0$ is the threshold in the finite-smoothness KAM
    theorem used in \cite{xinze2026most}.
\end{enumerate}

Then, for every $I_0\in D_\gamma$, the unperturbed torus
\begin{equation*}
    \mathcal T_{I_0}^{0}
    :=
    \mathbb T^n\times\{I_0\}
\end{equation*}
persists, up to a small deformation, as an invariant KAM torus
$\mathcal T_{I_0}^{P}$ of
\eqref{eq:deterministic-nearly-integrable-KAM-system}.

Moreover, for every $x_0\in\mathcal T_{I_0}^{P}$ and every
$\epsilon>0$, the deterministic Hamiltonian trajectory starting from
$x_0$, restricted to $[0,T]$, is a global minimizer of the
corresponding fractional Onsager--Machlup functional. Hence it is a
most probable evolution path of
\eqref{eq:fractional-stochastic-KAM-system}. Accordingly, the
deterministic KAM tori persist under fractional stochastic
perturbations as families of most probable evolution paths.
\end{theorem}

\begin{proof}
Under the stated regularity, non-degeneracy, Diophantine, and
smallness assumptions, the finite-smoothness KAM theorem used in
\cite{xinze2026most} yields the deformed invariant torus
$\mathcal T_{I_0}^{P}$ for every $I_0\in D_\gamma$.

By \eqref{eq:KAM-divergence-free}, the Hamiltonian drift has
identically vanishing divergence. Corollary~\ref{cor: spatially
constant divergence} therefore implies that, for every initial point
and every $\epsilon>0$, the corresponding deterministic Hamiltonian
trajectory is a global minimizer of the fractional
Onsager--Machlup functional.

If $x_0\in\mathcal T_{I_0}^{P}$, invariance of
$\mathcal T_{I_0}^{P}$ implies that this deterministic trajectory
remains on the torus. It is therefore a most probable evolution path
of \eqref{eq:fractional-stochastic-KAM-system}.
\end{proof}

\begin{remark}
The preceding theorem concerns persistence in the most probable
evolution-path sense. It neither asserts that sample paths of
\eqref{eq:fractional-stochastic-KAM-system} remain on the deterministic
KAM tori nor constructs random invariant tori. Rather, the
deterministic KAM tori persist as invariant sets of the perturbed
deterministic Hamiltonian system, and their trajectories are selected
as most probable evolution paths by the fractional
Onsager--Machlup functional.
\end{remark}

\subsection{Large deviations near the KAM tori}

For the small-noise asymptotics, define
\begin{equation*}
    \sigma_t
    :=
    \begin{pmatrix}
        \sigma_\theta(t) & 0
        \\
        0 & \sigma_I(t)
    \end{pmatrix}.
\end{equation*}
We work with the lifted process
\begin{equation*}
    X^\epsilon
    =
    (\theta^\epsilon,I^\epsilon)
\end{equation*}
on $\mathbb R^{2n}$ and set
\begin{equation*}
    \mathcal X
    :=
    C([0,T];\mathbb R^{2n}),
\end{equation*}
equipped with the uniform topology.

Recall that $\{X^\epsilon\}_{\epsilon>0}$ satisfies a large deviation
principle with speed $\epsilon^{-1}$ and rate function
$I_H:\mathcal X\to[0,\infty]$ if, for every Borel set
$A\subset\mathcal X$,
\begin{equation}
\label{eq:LDP-definition-Hamiltonian}
\begin{aligned}
    -\inf_{\phi\in A^\circ}I_H(\phi)
    &\leq
    \liminf_{\epsilon\to0}
    \epsilon\log\mathbb P(X^\epsilon\in A)
    \\
    &\leq
    \limsup_{\epsilon\to0}
    \epsilon\log\mathbb P(X^\epsilon\in A)
    \leq
    -\inf_{\phi\in\overline A}I_H(\phi).
\end{aligned}
\end{equation}
The rate function is called good if
\begin{equation*}
    \left\{
        \phi\in\mathcal X:
        I_H(\phi)\leq M
    \right\}
\end{equation*}
is compact for every $M<\infty$. Informally,
\begin{equation*}
    \mathbb P(X^\epsilon\approx\phi)
    \approx
    \exp\left\{
        -\frac{I_H(\phi)}{\epsilon}
    \right\}
\end{equation*}
on the logarithmic scale.

Applying the large deviation result of \cite{FanYuYuan2023} to the
present distribution-independent system gives the following theorem.
The small parameter in \cite{FanYuYuan2023} is denoted by $\delta$,
with noise magnitude $\delta^H$ and large-deviation scale
$\delta^{2H}$. Setting
\begin{equation*}
    \epsilon=\delta^{2H}
\end{equation*}
gives the noise magnitude $\sqrt{\epsilon}$ and the scale $\epsilon$
used here.

\begin{theorem}[Large deviations for the fractional stochastic Hamiltonian system]
\label{cor:fractional-Hamiltonian-LDP}
Let $1/2<H<1$, and let $X^\epsilon$ solve
\eqref{eq:fractional-stochastic-KAM-system} with
$X_0^\epsilon=x_0$. Suppose that $b_{\mathscr H}$ and $\sigma$
satisfy Assumption~\ref{ass:A}, with $\sigma_t$ nondegenerate for
every $t\in[0,T]$.

Then $\{X^\epsilon\}_{\epsilon>0}$ satisfies a large deviation
principle on $\mathcal X$ with speed $\epsilon^{-1}$ and good rate
function
\begin{equation}
\label{eq:Hamiltonian-rate-function}
    I_H(\phi)
    =
    \begin{cases}
        \displaystyle
        \frac12
        \int_0^T
        \left|
            (K_H^\sigma)^{-1}
            \left(
                \phi_\cdot-x_0
                -
                \int_0^\cdot
                b_{\mathscr H}(\phi_s)\,ds
            \right)(t)
        \right|^2\,dt,
        &
        \begin{aligned}
        &\phi_\cdot-x_0
        -
        \displaystyle\int_0^\cdot
        b_{\mathscr H}(\phi_s)\,ds
        \in\mathcal H_H^\sigma,
        \end{aligned}
        \\[4ex]
        +\infty,
        &\text{otherwise}.
    \end{cases}
\end{equation}
On its effective domain, $I_H$ coincides with the fractional kinetic
term of the Onsager--Machlup functional defined in
\eqref{eq:fractional-kinetic-term}.
\end{theorem}

\begin{remark}[Interpretation of the rate function]
\label{rem:Hamiltonian-LDP-interpretation}
The rate function has a natural control interpretation. A path $\phi$
has finite rate precisely when it can be written as
\begin{equation*}
    \phi_t
    =
    x_0
    +
    \int_0^t b_{\mathscr H}(\phi_s)\,ds
    +
    (K_H^\sigma u)(t)
\end{equation*}
for some $u\in L^2([0,T];\mathbb R^{2n})$. In that case,
\begin{equation*}
    I_H(\phi)
    =
    \frac12
    \|u\|_{L^2([0,T];\mathbb R^{2n})}^2.
\end{equation*}
Thus, $I_H(\phi)$ measures the Cameron--Martin control energy required
for the small-noise system to follow $\phi$.

Let $\bar X$ be the deterministic Hamiltonian trajectory starting
from $x_0$. Then
\begin{equation*}
    I_H(\phi)\geq0,
    \qquad
    I_H(\phi)=0
    \quad\Longleftrightarrow\quad
    \phi=\bar X.
\end{equation*}
Since $I_H$ is good, for every $\rho>0$,
\begin{equation*}
    c_\rho
    :=
    \inf_{\|\phi-\bar X\|_\infty\geq\rho}I_H(\phi)
    >
    0.
\end{equation*}
Consequently,
\begin{equation}
\label{eq:LDP-concentration-deterministic-path}
    \limsup_{\epsilon\to0}
    \epsilon
    \log
    \mathbb P
    \left(
        \|X^\epsilon-\bar X\|_\infty\geq\rho
    \right)
    \leq
    -c_\rho.
\end{equation}
Thus, over every fixed time interval, the stochastic trajectory
concentrates exponentially near $\bar X$ as $\epsilon\to0$.

If $x_0\in\mathcal T_{I_0}^{P}$, then
$\bar X_t\in\mathcal T_{I_0}^{P}$ for all $t$, and hence
\begin{equation*}
    \sup_{0\leq t\leq T}
    \operatorname{dist}
    \left(
        X_t^\epsilon,\mathcal T_{I_0}^{P}
    \right)
    \leq
    \|X^\epsilon-\bar X\|_\infty.
\end{equation*}
Therefore, \eqref{eq:LDP-concentration-deterministic-path} also gives
exponential concentration near the perturbed KAM torus.

For $H>1/2$, this complements
Theorem~\ref{thm:fractional-most-probable-KAM-tori}: the large
deviation principle describes the exponential concentration of
small-noise trajectories near the deterministic motion on the KAM
torus, while the Onsager--Machlup functional selects that
deterministic motion as a most probable evolution path.
\end{remark}

\section{Numerical experiments}
In this section, we present two examples illustrating the main results.
The first demonstrates the persistence and loss of most probable paths
as the noise intensity varies, together with their long-time breakdown
in the regime $H>1/2$. The second provides a geometric visualization
of a preserved KAM torus and the associated deterministic and
stochastic trajectories.

\subsection{Most probable paths and long-time instability of a periodic orbit} 
We illustrate the preceding results with a concrete example. Consider the drift
\begin{equation*}
    b(x,y)
    =
    \begin{pmatrix}
        \tfrac{1}{2}G(\rho)x-y\\[1mm]
        \tfrac{1}{2}G(\rho)y+x
    \end{pmatrix},
    \qquad
    \rho=x^2+y^2,
\end{equation*}
where
\begin{equation*}
    G(\rho)
    =
    -(\rho-1)+(\rho-1)^2-\frac{4}{3}(\rho-1)^3,
\end{equation*}
and let the diffusion matrix be
\begin{equation*}
    \sigma_s
    =
    \begin{pmatrix}
        1 & 0\\
        0 & 2+\cos s
    \end{pmatrix}.
\end{equation*}
The corresponding deterministic system admits the $2\pi$-periodic orbit
\begin{equation*}
    \bar X_s
    =
    \begin{pmatrix}
        \cos s\\
        \sin s
    \end{pmatrix}.
\end{equation*}

A direct calculation gives
\begin{equation*}
    \nabla\cdot b(x,y)
    =
    G(\rho)+\rho G'(\rho)
    =
    -1-(\rho-1)^2-\frac{16}{3}(\rho-1)^3.
\end{equation*}
Writing $\nabla\cdot b=F(\rho)$, we have
\begin{equation*}
    F'(1)=0,
    \qquad
    F''(1)=-2.
\end{equation*}
Since $\rho=x^2+y^2$,
\begin{equation*}
    \nabla^2(\nabla\cdot b)(x,y)
    =
    2F'(\rho)I_2
    +
    4F''(\rho)
    \begin{pmatrix}
        x^2 & xy\\
        xy & y^2
    \end{pmatrix}.
\end{equation*}
Consequently,
\begin{equation*}
    \nabla(\nabla\cdot b)(\bar X_s)=0,
\end{equation*}
and
\begin{equation}
\label{eq:hessian-divergence-along-orbit}
    \nabla^2(\nabla\cdot b)(\bar X_s)
    =
    -8\,\bar X_s\bar X_s^T.
\end{equation}

We next identify the zero modes of the linearized operator
\begin{equation*}
    \mathcal L
    =
    \frac{d}{ds}-\nabla b(\bar X_s).
\end{equation*}
Introduce the $2\pi$-periodic orthogonal matrix
\begin{equation*}
    P(s)
    =
    \begin{pmatrix}
        \cos s & -\sin s\\
        \sin s & \cos s
    \end{pmatrix}.
\end{equation*}
For $\eta_s=P(s)z_s$, one obtains
\begin{equation*}
    \mathcal L\eta_s
    =
    P(s)
    \left(
        \frac{d}{ds}-R
    \right)z_s,
    \qquad
    R
    =
    \begin{pmatrix}
        -1 & 0\\
        0 & 0
    \end{pmatrix}.
\end{equation*}
Thus, the right zero mode is
\begin{equation*}
    v_s
    =
    P(s)
    \begin{pmatrix}
        0\\
        1
    \end{pmatrix}
    =
    \begin{pmatrix}
        -\sin s\\
        \cos s
    \end{pmatrix}
    =
    \dot{\bar X}_s.
\end{equation*}
Since $P(s)$ is orthogonal and $R$ is symmetric, the corresponding left zero
mode coincides with $v_s$:
\begin{equation*}
    \omega_s=v_s.
\end{equation*}

The diffusion matrix $\sigma_s$ is uniformly nondegenerate, and
\begin{equation*}
    \int_0^{2\pi}\sigma_s^T\omega_s\,ds
    =
    \begin{pmatrix}
        0\\
        \pi
    \end{pmatrix}.
\end{equation*}
Taking
\begin{equation*}
    q
    =
    \begin{pmatrix}
        1\\
        0
    \end{pmatrix},
\end{equation*}
we obtain
\begin{equation*}
    q^T
    \int_0^{2\pi}\sigma_s^T\omega_s\,ds
    =
    0.
\end{equation*}
Hence, the Fredholm solvability condition is satisfied, and the periodic problem
\begin{equation*}
    \mathcal L\mu_s
    =
    \sigma_s q
    =
    \begin{pmatrix}
        1\\
        0
    \end{pmatrix}
\end{equation*}
admits a $2\pi$-periodic solution.

Indeed, writing $\mu_s=P(s)z_s$, the equation reduces to
\begin{equation*}
    z_s'-Rz_s
    =
    P(s)^T
    \begin{pmatrix}
        1\\
        0
    \end{pmatrix}
    =
    \begin{pmatrix}
        \cos s\\
        -\sin s
    \end{pmatrix}.
\end{equation*}
One periodic solution is
\begin{equation*}
    z_s
    =
    \begin{pmatrix}
        \dfrac{1}{2}(\cos s+\sin s)\\[1mm]
        \cos s
    \end{pmatrix}.
\end{equation*}
Since $\bar X_s=P(s)(1,0)^T$, its radial component is
\begin{equation*}
    \bar X_s^T\mu_s
    =
    \frac{1}{2}(\cos s+\sin s),
\end{equation*}
which does not vanish identically. Therefore, by
\eqref{eq:hessian-divergence-along-orbit},
\begin{equation*}
\begin{aligned}
    \int_0^{2\pi}
    \mu_s^T
    \nabla^2(\nabla\cdot b)(\bar X_s)
    \mu_s\,ds
    &=
    -8
    \int_0^{2\pi}
    \left|
        \bar X_s^T\mu_s
    \right|^2
    ds\\
    &=
    -4\pi
    <0.
\end{aligned}
\end{equation*}
Thus, the periodic response generated by the forcing $\sigma_s q$ has a
nontrivial radial component and produces a strictly negative contribution to
the second variation.

Finally, we numerically investigate the most probable transition paths
in the Brownian regime $H=1/2$ over $[0,2\pi]$. We consider the noise
intensities $\epsilon=0.01$ and $\epsilon=0.3$ and compare the
deterministic periodic orbit, sample paths satisfying the endpoint
constraint, and the numerically computed most probable transition
path.

As shown in Figure~\ref{fig:noise-comparison}, the most probable path
remains close to the deterministic periodic orbit for small noise,
whereas a pronounced deviation is observed for larger noise. This
behavior is consistent with the persistence of the deterministic path
in the small-noise regime and the possible loss of its local
minimality as the noise intensity increases.

For the long-time problem, directly plotting the periodically extended
trajectory becomes uninformative as the number of periods increases,
because the repeated revolutions make the resulting phase portraits
difficult to distinguish. We therefore illustrate the long-time loss
of positive definiteness by directly computing the second variation
along a family of admissible perturbations.

Extend the periodic response mode $\mu$ periodically to $[0,2\pi N]$
and define
\begin{equation}
\label{eq:numerical-perturbation}
    \eta_N(s)
    :=
    \chi\left(\frac{s}{N}\right)\mu_s,
    \qquad
    \chi(u):=1-\cos u,
    \qquad
    0\leq s\leq2\pi N.
\end{equation}
Since $\chi(0)=\chi(2\pi)=0$, we have
\begin{equation*}
    \eta_N(0)=\eta_N(2\pi N)=0,
\end{equation*}
so $\eta_N$ is an admissible fixed-endpoint perturbation. We then
numerically evaluate
\begin{equation*}
    \delta^2J_{\epsilon,N}(\bar X)[\eta_N,\eta_N]
    =
    \frac{1}{\epsilon}\mathcal I_H^N(\eta_N)
    -
    \frac{d_H}{2}\mathcal V^N(\eta_N).
\end{equation*}

Fixing the noise intensity at $\epsilon=0.1$, we perform the computation
for $H=0.7$, $H=0.5$, and $H=0.3$. The results in
Figure~\ref{fig:long-time-second-variation} illustrate the
Hurst-dependent behavior predicted by the theoretical analysis. For
$H=0.7$, the second variation eventually becomes negative, so
$\eta_N$ provides a negative direction and demonstrates the loss of
local minimality. By contrast, for $H=0.5$ and $H=0.3$, the second
variation remains positive over the computed range of periods.

\begin{figure}[htbp]
    \centering

    \begin{subfigure}[t]{0.62\textwidth}
        \centering
        \includegraphics[width=\textwidth]{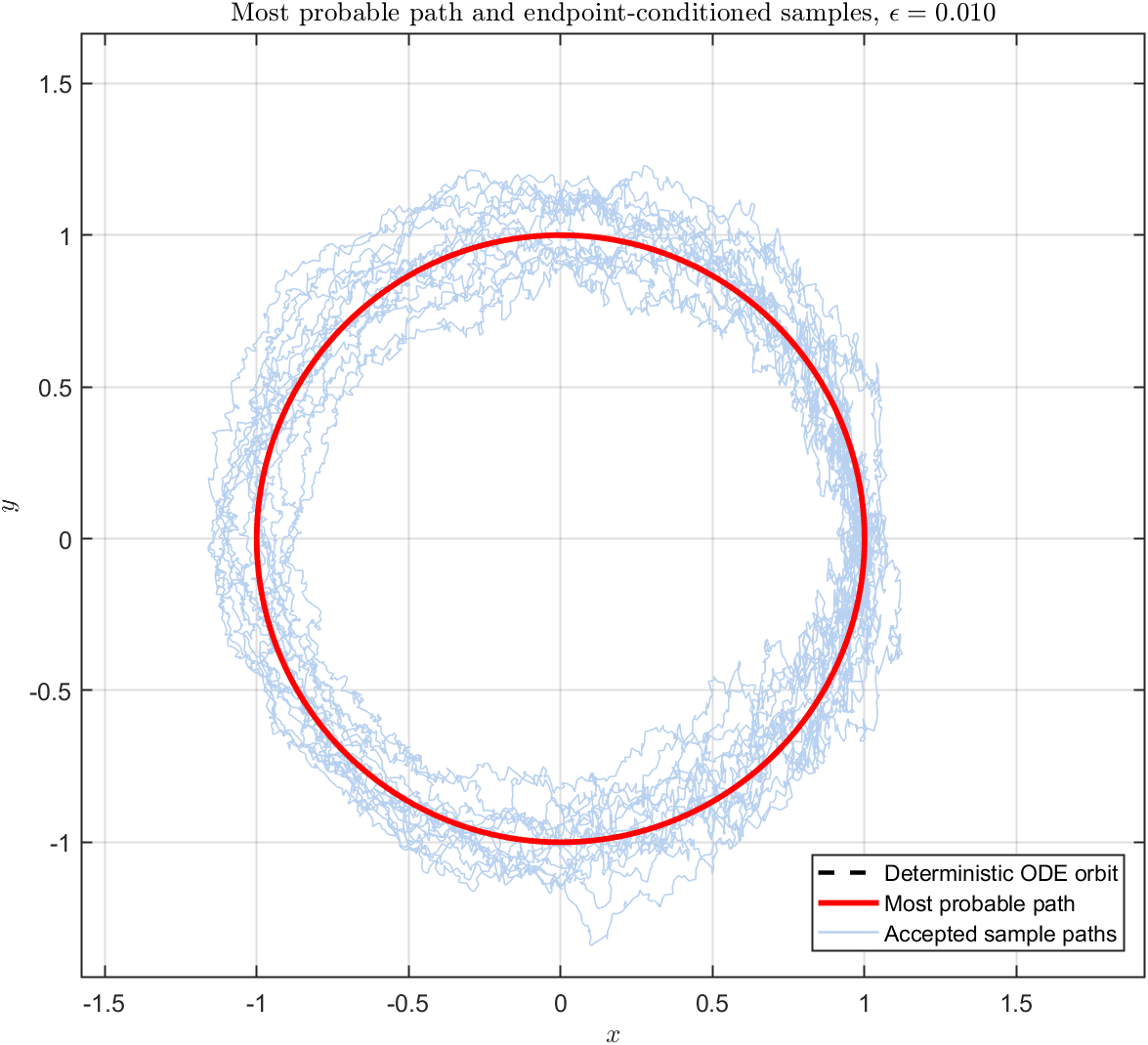}
        \caption{Small-noise regime, $\epsilon=0.01$.}
        \label{fig:small-noise}
    \end{subfigure}

    \vspace{2mm}

    \begin{subfigure}[t]{0.62\textwidth}
        \centering
        \includegraphics[width=\textwidth]{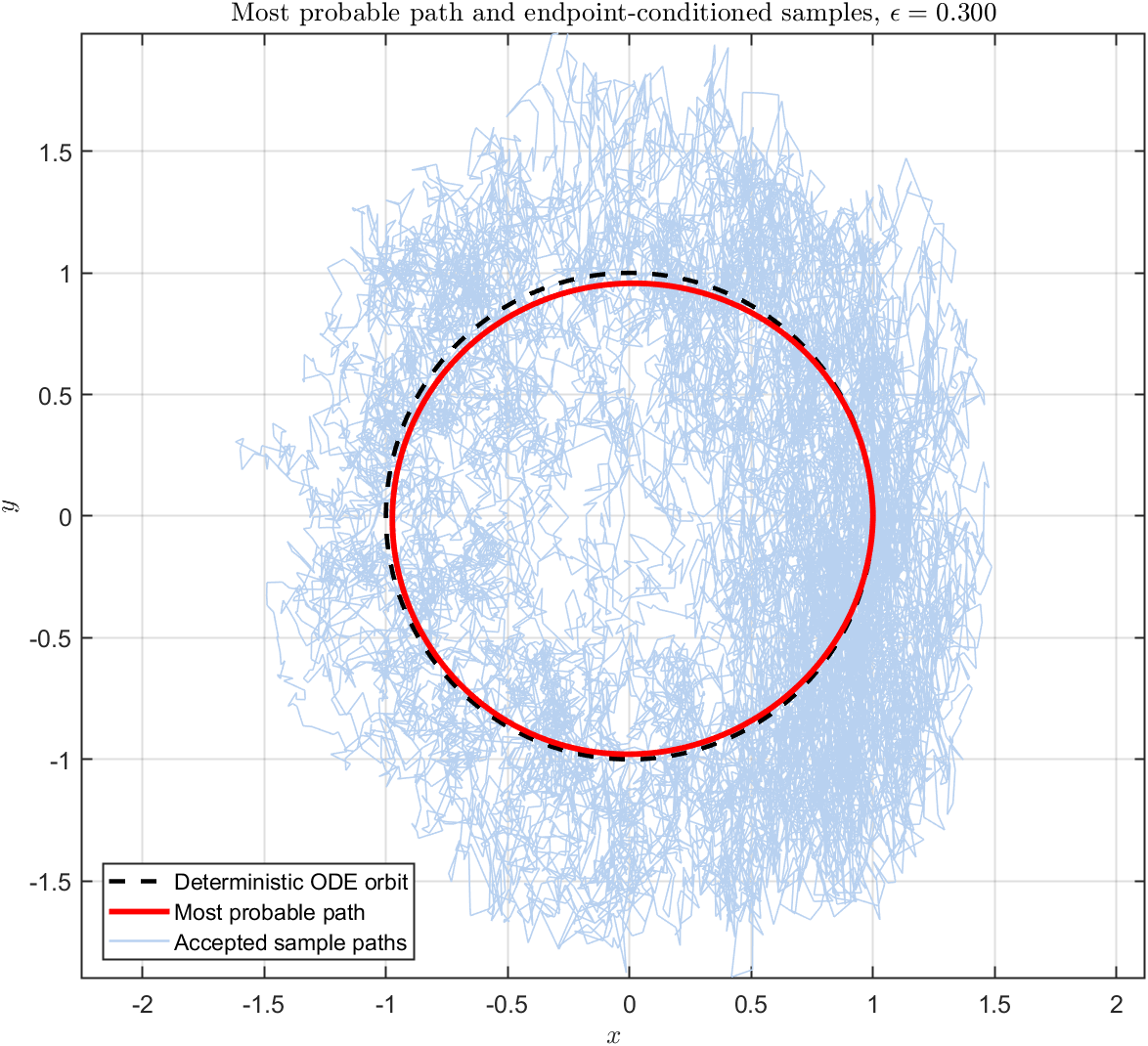}
        \caption{Larger-noise regime, $\epsilon=0.3$.}
        \label{fig:large-noise}
    \end{subfigure}

    \caption{Most probable transition paths in the Brownian regime
    $H=1/2$ over $[0,2\pi]$. For small noise, the numerically computed
    most probable path remains close to the deterministic periodic
    orbit. At the larger noise intensity, a pronounced deviation from
    the deterministic orbit is observed.}
    \label{fig:noise-comparison}
\end{figure}

\begin{figure}[p]
    \centering

    \begin{subfigure}[t]{0.82\textwidth}
        \centering
        \includegraphics[width=\textwidth]{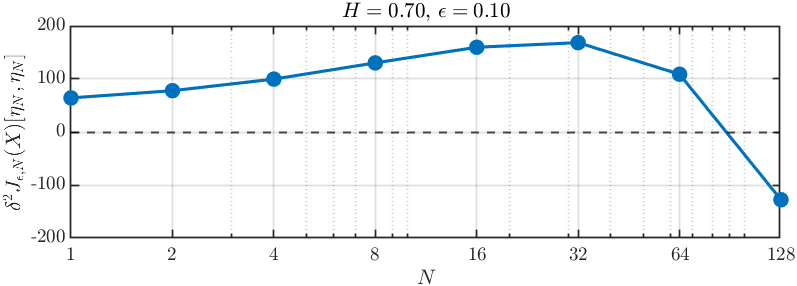}
        \caption{}
        \label{fig:second-variation-eps01}
    \end{subfigure}

    \vspace{2mm}

    \begin{subfigure}[t]{0.82\textwidth}
        \centering
        \includegraphics[width=\textwidth]{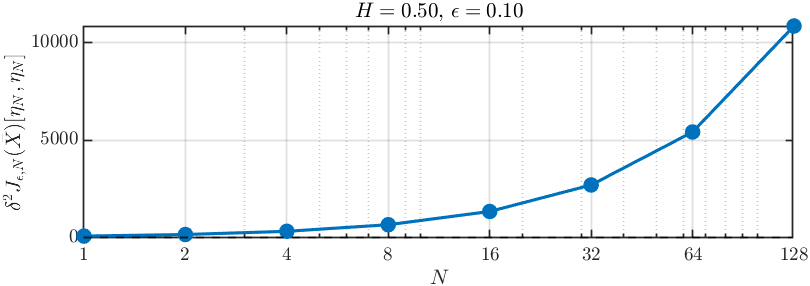}
        \caption{}
        \label{fig:second-variation-eps02}
    \end{subfigure}

    \vspace{2mm}

    \begin{subfigure}[t]{0.82\textwidth}
        \centering
        \includegraphics[width=\textwidth]{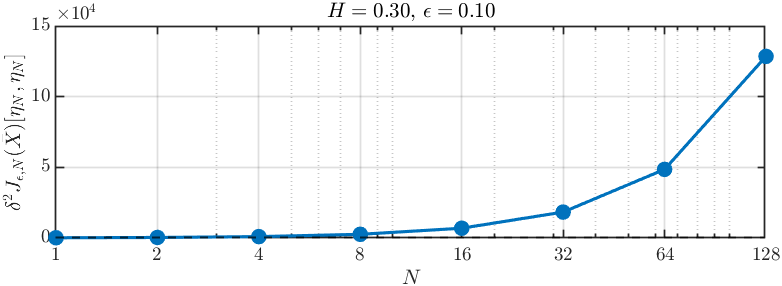}
        \caption{}
        \label{fig:second-variation-eps03}
    \end{subfigure}

    \caption{Second variation
    $\delta^2J_{\epsilon,N}(\bar X)[\eta_N,\eta_N]$ as a function of
    the number of periods $N$ for $H=0.7$. Negative values identify
    a destabilizing direction and therefore imply that the
    periodically extended deterministic orbit is not a local
    minimizer. The second variation need not vary monotonically with
    $N$ because of finite-scale interactions between the slowly
    varying envelope and the periodic response mode. Nevertheless, as
    $N$ increases, the linearly growing potential contribution
    eventually dominates the fractional kinetic contribution, leading
    to the loss of positive definiteness. Larger noise intensities
    generally lead to an earlier occurrence of this instability.}
    \label{fig:long-time-second-variation}
\end{figure}

\subsection{Most probable evolution paths on a KAM torus}

We illustrate the preceding result using an explicitly solvable
two-degree-of-freedom nearly integrable Hamiltonian system. Let
\begin{equation*}
    F(\theta)
    :=
    \sin\theta_1\sin\theta_2,
    \qquad
    J
    :=
    I-\delta\nabla F(\theta),
\end{equation*}
and consider
\begin{equation}
\label{eq:numerical-KAM-Hamiltonian}
    \mathscr H_\delta(\theta,I)
    :=
    \omega^T(J-I_0)
    +
    \frac{\kappa}{2}|J-I_0|^2,
    \qquad
    (\theta,I)\in\mathbb T^2\times\mathbb R^2.
\end{equation}
Equivalently,
\begin{equation*}
    \mathscr H_\delta(\theta,I)
    =
    H_0(I)+P_\delta(\theta,I),
\end{equation*}
where
\begin{equation*}
    H_0(I)
    :=
    \omega^T(I-I_0)
    +
    \frac{\kappa}{2}|I-I_0|^2
\end{equation*}
and
\begin{equation*}
\begin{aligned}
    P_\delta(\theta,I)
    &=
    -\delta
    \left[
        \omega+\kappa(I-I_0)
    \right]^T
    \nabla F(\theta)
    +
    \frac{\kappa\delta^2}{2}
    |\nabla F(\theta)|^2.
\end{aligned}
\end{equation*}
Thus, $P_\delta=O(\delta)$, while
\begin{equation*}
    \partial_I^2H_0=\kappa I_2
\end{equation*}
is nondegenerate whenever $\kappa\neq0$.

The transformation
\begin{equation*}
    J=I-\delta\nabla F(\theta)
\end{equation*}
is canonical. Consequently, the perturbed system possesses the
explicit invariant torus
\begin{equation}
\label{eq:explicit-deformed-KAM-torus}
    \mathcal T_\delta
    :=
    \left\{
        (\theta,I)\in\mathbb T^2\times\mathbb R^2:
        I=I_0+\delta\nabla F(\theta)
    \right\}.
\end{equation}
The dynamics on $\mathcal T_\delta$ is
\begin{equation*}
    \dot\theta=\omega,
    \qquad
    I=I_0+\delta\nabla F(\theta).
\end{equation*}
We take
\begin{equation*}
    \omega=
    \begin{pmatrix}
        1\\ \sqrt{2}
    \end{pmatrix},
    \qquad
    \delta=0.18,
    \qquad
    \kappa=0.30.
\end{equation*}
Since the components of $\omega$ are rationally independent, the
corresponding deterministic trajectory is quasi-periodic on
$\mathcal T_\delta$.

We further consider the fractional stochastic perturbation
\begin{equation}
\label{eq:numerical-stochastic-KAM-system}
\left\{
\begin{aligned}
    d\theta_t
    &=
    \left[
        \omega+\kappa(J_t-I_0)
    \right]dt
    +
    \sqrt{\epsilon}\,
    \sigma_\theta(t)\,dB_{\theta,t}^H,
    \\
    dI_t
    &=
    \delta\nabla^2F(\theta_t)
    \left[
        \omega+\kappa(J_t-I_0)
    \right]dt
    +
    \sqrt{\epsilon}\,
    \sigma_I(t)\,dB_{I,t}^H,
    \\
    J_t
    &=
    I_t-\delta\nabla F(\theta_t),
\end{aligned}
\right.
\end{equation}
where $B_\theta^H$ and $B_I^H$ are independent two-dimensional
fractional Brownian motions. In the simulation, we set
\begin{equation*}
    H=0.7,
    \qquad
    \epsilon=0.02,
\end{equation*}
and use the time-dependent diffusion matrices
\begin{equation*}
    \sigma_\theta(t)
    =
    \operatorname{diag}
    \left(
        0.10(1+0.20\sin t),
        0.08(1+0.20\cos t)
    \right)
\end{equation*}
and
\begin{equation*}
    \sigma_I(t)
    =
    \operatorname{diag}
    \left(
        0.045(1+0.20\cos t),
        0.045(1+0.20\sin t)
    \right).
\end{equation*}

Since the phase space is four-dimensional, the invariant torus and
the trajectories are displayed through a three-dimensional
projection. The red curve represents a trajectory of the unperturbed
integrable system, the blue curve represents the corresponding
quasi-periodic trajectory on the deformed invariant torus
$\mathcal T_\delta$, and the light-gray curves are sample trajectories
of \eqref{eq:numerical-stochastic-KAM-system}. Only the upper half of
the projected torus is displayed to avoid obscuring the trajectories.

As shown in Figure~\ref{fig:numerical-KAM-torus}, the deterministic
perturbation deforms the original invariant torus while preserving its
quasi-periodic dynamics. The stochastic sample paths generally leave
the deterministic torus, since the torus is not pathwise invariant
under the stochastic dynamics. Nevertheless, the deterministic
trajectory on $\mathcal T_\delta$ is selected as a most probable
evolution path by the Onsager--Machlup functional. The figure therefore
illustrates variational persistence rather than the existence of a
random invariant torus.

\begin{figure}[htbp]
    \centering
    \includegraphics[width=0.78\textwidth]{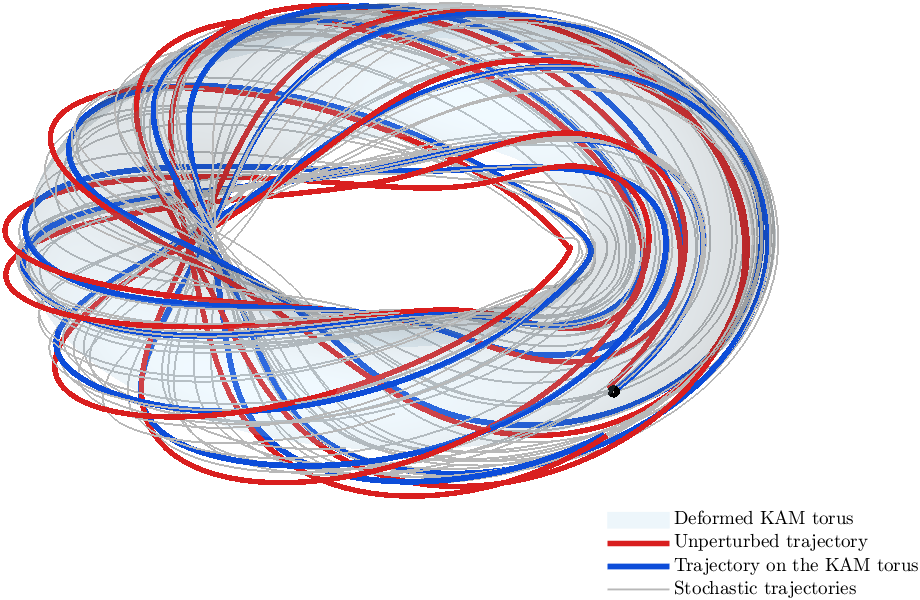}
    \caption{
        Three-dimensional projection of the invariant torus and the
        associated trajectories for $H=0.7$ and $\epsilon=0.02$.
        The red curve is an unperturbed quasi-periodic trajectory, the
        blue curve is the corresponding trajectory on the deformed
        KAM torus, and the light-gray curves are sample trajectories
        under fractional stochastic perturbations. The transparent
        surface represents the upper half of the projected deformed
        invariant torus.
    }
    \label{fig:numerical-KAM-torus}
\end{figure}

\section*{Funding}

This work was supported by the National Key R\&D Program of China
[grant number 2023YFA1009200], the National Key Project of the National
Natural Science Foundation of China [grant number 12531009], and the
National Natural Science Foundation of China
[grant numbers 12471183 and 12071175].

        \clearpage

\bibliographystyle{plain}
\bibliography{cite.bib}

\end{document}